\documentclass[11pt,reqno]{amsart}

\usepackage{ifthen}
\usepackage{mathptmx}
\usepackage{amsmath}
\usepackage{amscd}
\usepackage{amssymb}
\usepackage{amsthm}
\usepackage{xspace}
\usepackage[all,tips]{xy}
\usepackage[dvips]{graphicx}
\usepackage{verbatim}
\usepackage{syntonly}
\usepackage{hyperref}
\usepackage{amsmath, amsthm, graphics, amssymb,color, epsfig,url}

\newcommand{\showcomments}{yes}

\newsavebox{\commentbox}
{\ifthenelse{\equal{\showcomments}{yes}}%
{\footnotemark
        \begin{lrbox}{\commentbox}
        \begin{minipage}[t]{1.25in}\raggedright\sffamily\tiny
        \footnotemark[\arabic{footnote}]}
{\begin{lrbox}{\commentbox}}}%
{\ifthenelse{\equal{\showcomments}{yes}}%
{\end{minipage}\end{lrbox}\marginpar{\usebox{\commentbox}}}
{\end{lrbox}}}

\theoremstyle{plain}
\newtheorem{thm}{Theorem}[section]
\newtheorem{cor}[thm]{Corollary}
\newtheorem{prop}[thm]{Proposition}
\newtheorem{lemma}[thm]{Lemma}
\newtheorem{ques}{Question}

\theoremstyle{definition}
\newtheorem{rem}{Remark}

\DeclareMathOperator{\rel}{\;rel\; }
\DeclareMathOperator{\rank}{rank}
\DeclareMathOperator{\Spec}{Spec}
\DeclareMathOperator{\Frac}{Frac}
\DeclareMathOperator{\Gal}{Gal}
\DeclareMathOperator{\trace}{trace}
\DeclareMathOperator{\red}{red}

\DeclareMathOperator{\SL}{SL} \DeclareMathOperator{\PSL}{PSL}
\DeclareMathOperator{\GL}{GL} \DeclareMathOperator{\PGL}{PGL}

\DeclareMathOperator{\SU}{SU}

\DeclareMathOperator{\rad}{rad}
\DeclareMathOperator{\Hom}{Hom}
\DeclareMathOperator{\Stab}{Stab}
\DeclareMathOperator{\rk}{rk}

\newcommand{\op}{\oplus}

\newcommand{\N}{\ensuremath{\mathbb{N}}}
\newcommand{\Q}{\ensuremath{\mathbb{Q}}}
\newcommand{\R}{\ensuremath{\mathbb{R}}}
\newcommand{\Z}{\ensuremath{\mathbb{Z}}}
\newcommand{\C}{\ensuremath{\mathbb{C}}}

\newcommand{\F}{\ensuremath{\mathbb{F}}}

\newcommand{\g}{\mathfrak g} 
\renewcommand{\sl}{\mathfrak sl} 

\newcommand{\m}{\mathfrak{m}}

\newcommand{\Lt}{\mathcal{L}}
\newcommand{\Lf}{\mathcal{B}}

\title{Linear groups with Borel's property}

\author{Khalid Bou-Rabee}
\email{khalid.math@gmail.com}
\address{School of Mathematics \\
University of Minnesota--Twin Cities \\
Minneapolis, MN 55455\\
U.S.A.}

\author{Michael Larsen}
\email{mjlarsen@indiana.edu}
\address{Department of Mathematics\\
    Indiana University \\
    Bloomington, IN 47405\\
    U.S.A.}

\thanks{Michael Larsen was partially supported by NSF grant DMS-1101424 and a Simons Fellowship.}

\begin{document}
\maketitle
\begin{abstract}When does Borel's theorem on free subgroups of semisimple groups generalize to other groups?
We initiate a systematic study of this question and find positive and negative answers for it. 
In particular, we fully classify fundamental groups of surfaces and von Dyck groups that satisfy Borel's theorem.
Further, as a byproduct of this theory, we make headway on a question of Breuillard, Green, Guralnick, and Tao concerning double word maps.
\end{abstract}

\tableofcontents

\section{Introduction}

Let $\Gamma$ be a group.
\emph{What group theoretic properties of $\Gamma$ can we infer from the flexibility of its representation theory?}
To systematically approach this basic question, we focus on the following property: 
$\Gamma$ has \emph{Borel's property} if for every connected semisimple group $G$, every proper subvariety $V$ of $G$, and every nontorsion $\gamma\neq 1$ in $\Gamma$, there exists a homomorphism $\phi\colon \Gamma\to G(\C)$ such that $\phi(\gamma)\not\in V(\C)$.
Loosely speaking, groups with Borel's property have so many representations that not only can every element be detected in any semisimple group, but any element can be made to miss every proper subvariety.

\begin{ques} \label{MainQuestion}
Which classes of groups have Borel's property?
\end{ques}

\noindent
In 1983, Armand Borel demonstrated the remarkable fact that free groups have Borel's property.
Question \ref{MainQuestion} also sheds light on double-word maps (see \S \ref{ApplicationsSection}) and, in fact, gives a partial answer to \cite[Problem 2]{BGGT10}. That is, we obtain

\begin{thm} \label{MainAnswerThm}
Let $w_1$ be a word of the form $[x_1, x_2] \cdots [x_{2k-1},x_{2k}]$ where $k \ge 2$.
Let $w_2$ be a word not in the normal closure of $w_1$.
Then the double-word map defined by $w_1, w_2$ is dominant.
\end{thm}

\noindent We hope that a continued study of our new theory will help obtain a full answer to  \cite[Problem 2]{BGGT10}.

With this hope in mind, the remainder of our main results revolve around applications of tools we have developed for determining whether a group has Borel's property.
Let $\Lt$ be the class of groups that satisfies Borel's property.
Let $\Lf$ be the class of torsion-free groups in $\Lt$.
We start with a complete classification of fundamental groups of surfaces that are in $\Lt$, which  
indicates that our new line of study is not an empty theory.

\begin{thm} \label{SurfaceGroupTheorem}
Let $S$ be a compact surface without boundary.
Then $\pi_1(S, \cdot)$ is in $\Lt$ if and only if $S$ is not the Klein bottle.
In particular, $\pi_1(S, \cdot)$ is in $\Lf$ if and only if $S$ is neither the Klein bottle nor the real projective plane.
\end{thm}

The examples in Theorem~\ref{SurfaceGroupTheorem} are handled in different parts of the paper.
The Klein bottle group is handled in Corollary~\ref{KleinBottleGroupCorollary} in \S \ref{NotAlmostGFreeGroupsSection}.
In studying this case, we discovered a Tits alternative for $\Lf$:

\begin{thm} \label{ContainsFreeSubgroupTheorem}
Let $\Gamma$ be a finitely generated group that is in $\Lf$.
Then $\Gamma$ contains a nonabelian free group or is a free abelian group.
\end{thm}

\noindent
See \S \ref{LinearSection} for a proof.
Since fundamental groups of oriented surfaces inject into a direct product of free groups  \cite{B62}, they are in $\Lf$ (see Lemma \ref{DirectProductLemma} in \S \ref{BasicTheorySection}), and the same can be said for connected sums of four or more real projective planes.
We are left with the fundamental group of the connected sum of three projective planes.
This group, which has presentation
$$
\pi_1 :=\left< a, b, c: a^2 b^2 c^2 = 1 \right>,
$$
\noindent does not inject into a direct product of free groups \cite{MR0162838} (for another example of a group in $\Lf$ that is not residually free, see Theorem~\ref{GFreeButNotRFTheorem} in \S \ref{DeterminingGFreeSection}).
Handling $\pi_1$ requires new machinery that we develop in \S \ref{DeterminingGFreeSection} (c.f. Proposition~\ref{ThreeProjectivePlanesProp}).
In total, our proof of Theorem \ref{SurfaceGroupTheorem}, and all the results of this paper, involve a fusion of ideas and methods from algebraic geometry, number theory, differential geometry, finite group theory, combinatorial group theory, and the theory of linear algebraic groups.

In \S \ref{BasicTheorySection}, we show that in the context of finitely generated groups
\begin{equation}
\label{chain}
\textrm{residually free} \subseteq \Lf \subseteq \Lt \cap \textrm{linear} \subseteq \Lt.
\end{equation}
The group $\pi_1$, discussed above, shows that the first containment is strict.
The last two containments are shown to be strict in \S \ref{BasicTheorySection}.
Significant examples of linear groups in $\Lt$ which are not in $\Lf$ are supplied by the next two theorems.

\begin{thm} \label{SevenSevenTheorem}
For $\ell\equiv 1$ (mod $3$) prime, the group $\Z/\ell * \Z/\ell$ is in $\Lt$.
\end{thm}

\begin{thm} \label{QuadTheorem}
Let $\ell\ge 19$ be a prime that is $\equiv 1$ (mod $3$). The group 
$$\left< x, y, z, t : x^\ell = y^\ell = z^\ell = t^\ell = xyzt \right>$$
is in $\Lt$.
\end{thm}

\noindent
The proofs of Theorems  \ref{SevenSevenTheorem} and \ref{QuadTheorem} appear in \S \ref{weakGSection}.
The proof of Theorem \ref{SevenSevenTheorem} relies on a delicate strengthening of Borel's original proof.
The proof of the latter theorem follows a similar track while relying, in addition, on a method developed by Avraham Aizenbud and Nir Avni \cite{AA}
and a new character theory estimate established in Appendix \ref{CharacterBoundsAppendix}.
The three previous theorems might lead to some hope that all Fuchsian groups are in $\Lt$.  However, this is certainly not the case:

\begin{thm} \label{DihedralGroupTheorem}
The infinite dihedral group is not in $\Lt$, and therefore no group containing infinitely many elements of order $2$,
such as a non-oriented Fuchsian group or an oriented Fuchsian group with an elliptic point of order $2$, can be in $\Lt$.
\end{thm}

\noindent
It turns out that even subgroups consisting of orientation-preserving isometries in triangle groups are not in $\Lt$.

\begin{thm} \label{TriangleGroupTheorem}
No von Dyck group is in $\Lt$.
\end{thm}
\noindent
The proofs of Theorems \ref{DihedralGroupTheorem} and \ref{TriangleGroupTheorem} appear in \S \ref{NotAlmostGFreeGroupsSection}.
It would be interesting to understand, in general, which Fuchsian groups are in $\Lt$.

This article is organized as follows. In \S \ref{BasicTheorySection} basic notions are defined, groups in $\Lf$ are shown to be linear, and a Tits alternative for $\Lf$ is established.
In \S \ref{DeterminingGFreeSection} and \S \ref{weakGSection} general methods for determining whether  a group is in $\Lf$ or $\Lt$, respectively, are described and applied to some examples. 
\S \ref{NotAlmostGFreeGroupsSection} discusses obstructions to $G$-freeness or almost freeness.
In \S \ref{ApplicationsSection} we give a partial  answer to a question of Emmanuel Breuillard, Ben Green, Robert Guralnick, and Terence Tao concerning double word maps.
For the convenience of the reader, Appendix \ref{AlgebraicGeometryBackgroundAppendix} collects some basic definitions and facts from algebraic geometry which are used in the
paper.
Appendix \ref{CharacterBoundsAppendix} gives some bounds on irreducible characters of certain finite groups that are used in \S \ref{weakGSection}.

\subsection*{Acknowledgements}

The first-named author gratefully acknowledges the hospitality and support given to him from the Ventotene 2013 conference for a week while he worked on some of the material in this paper.  
The first-named author is also grateful to Albert Marden for pointing him towards the work of Sepp\"al\"a and Sorvali.
The second-named author would like to acknowledge useful conversations with Nir Avni, Aner Shalev, and Pham Tiep.
Both authors are grateful to Benson Farb for giving helpful comments on an earlier draft.

\section{General theory} \label{BasicTheorySection}

In this section, we present some basic results and examples that we hope will cast light on Question \ref{MainQuestion}.
Before we begin, we need some notation that allows us to succinctly describe groups that satisfy Borel's property to varying degrees for a given linear algebraic group $G$.

\subsection{Notation and Terminology.}
For notation and terminology regarding algebraic geometry, see Appendix \ref{AlgebraicGeometryBackgroundAppendix}.  Here we note only that we do not assume that
our varieties are either irreducible or reduced, but for our purposes infinitesimal structure will never matter.
Algebraic groups will be assumed to be taken over $\C$ unless some other field is specified explicitly.

Let $\Gamma$ be a finitely generated group. 
Denote by $\Gamma^\bullet$ the set $\Gamma \setminus \{1 \}$.
Let $\mathcal{P}$ be a class of groups.
We say that $S \subseteq \Gamma$ is \emph{detected by $\mathcal{P}$} if there exists a homomorphism $\phi: \Gamma \to P$, $P \in \mathcal{P}$ such that $\phi(S) \cap \{ 1 \} = \emptyset$.
We say that $S \subseteq \Gamma$ is \emph{almost detected by $\mathcal{P}$} if there exists a homomorphism $\phi: \Gamma \to P$, $P \in \mathcal{P}$ such that $\phi(s)$ is nontorsion for every nontorsion $s \in S$.
In the case when $S = \{ \gamma \}$, we can sometimes say that \emph{$\gamma$ is (almost) detected by $\mathcal{P}$} instead of $\{ \gamma \}$ is (almost) detected by $\mathcal{P}$.
If every element in $\Gamma^\bullet$ is (almost) detected by $\mathcal{P}$ we say that $\Gamma$ is \emph{(almost) detectable by $\mathcal{P}$} or is \emph{(almost) residually $\mathcal{P}$}.

Let $G$ be a linear algebraic group defined over $\C$.
Let $V$ be a subvariety of $G$.
We say $S \subseteq \Gamma$ is \emph{detected by $G\rel V$} if there exists a homomorphism $\phi : \Gamma \to G(\C)$ such that $\phi(S) \cap V = \emptyset$.
If $S = \{ \gamma \}$ we sometimes say \emph{$\gamma$ is detected by $G\rel V$} instead of $S$ is detected by $G\rel V$.
If every element in $\Gamma^\bullet$ is detected by $G\rel V$ for every subvariety $V$ of $G$, then we say that $\Gamma$ is \emph{$G$-free}.
If every nontorsion element in $\Gamma$ is detected by $G\rel V$ for any subvariety $V$ of $G$, then we say that $\Gamma$ is \emph{almost $G$-free}.
Equivalently, $\Gamma$ is $G$-free (resp. almost $G$-free) if and only if the \emph{evaluation map} 
$$e_{G,\gamma}\colon \Hom(\Gamma,G(\C))\to G(\C)$$ 
has dense image for all $\gamma\neq 1$ (resp. all nontorsion $\gamma$).
In the definition of (almost) $G$-free, if the representations can be taken to be faithful, we say that the group is  \emph{(almost) $G$-faithful}.

\subsection{Some basic results}\label{BasicTheorySubSection}

We start by showing that groups in $\Lf$ are precisely those in $\Lt$ that are $G$-free for some connected semisimple $G$.

\begin{lemma} \label{TorsionFreeLemma}
Let $G$ be a connected semisimple algebraic group. Then any $G$-free group must be torsion-free. 
\end{lemma}

\begin{proof}
Let $\Gamma$ be a group with an element $\gamma$ of order $k$, where $k \in \Z^{>0}$.
For any map $\phi : \Gamma \to G(\C)$, we have $\phi (\gamma) \in  \{ A \in G : A^k = 1 \}$.
If $e_{G,x^k}$ denotes the $k$th power map on $G$, then $e_{G,x^k}^{-1}(G)$ is a proper closed subvariety of $G$,
and $\gamma$ cannot be detected by $G$ rel $e_{G,x^k}^{-1}(G)$.
\end{proof}

The following lemma is a slight strengthening of the fact that free groups are in $\Lf$.
While the result is known, we include a proof  for completeness.

\begin{lemma} \label{FreeLemma}
Finitely generated free groups are $G$-faithful for all connected semisimple linear algebraic $G$.
\end{lemma}

\begin{proof}
Let $F_d$ be a finitely generated free group of rank $d$.
Let $V$ be a subvariety of $G$, where $G$ is a connected semisimple linear algebraic $G$.
By Borel's Theorem \cite{B83}, we have that 
$$
L_\gamma := \{ \phi \in \Hom(F_d, G(\C)) : \phi( \gamma ) \notin V\cup\{1\} \}
$$
is nonempty for every $\gamma\in F_d^\bullet$.
Each $L_\gamma$ is consists of the set of closed points of a proper closed subvariety of $\Hom(F_d,G) \cong G^d$
and is therefore a closed subset of $G(\C)^d$ without interior points.
As $F_d$ is countable, 
$$\bigcap_{\gamma \in F_d^\bullet} L_\gamma$$
is nonempty by the Baire category theorem (Theorem~\ref{Baire}).
Any representation lying in this intersection is faithful and satisfies $\phi(F_d^\bullet) \cap V = \emptyset$, so $F$ is $G$-faithful for any semisimple $G$.
\end{proof}

The following lemmas are useful tools for constructing elements in $\Lt$ or $\Lf$.

\begin{lemma} \label{DetectableLemma}
Let $\mathcal{P}$ be a class of  (almost) $G$-free groups.
If $\Gamma$ is (almost) detectable by $\mathcal{P}$, then $\Gamma$ is (almost) $G$-free.
\end{lemma}

\begin{proof}
Let $\gamma \in \Gamma^\bullet$ be a given (torsion-free) element.
Let $V$ be an arbitrary subvariety of $G$.
Since $\Gamma$ is (almost) detectable by $\mathcal{P}$, there exists a homomorphism $\phi: \Gamma \to P \in \mathcal{P}$ with $\phi(\gamma) \neq 1$ ($\phi(\gamma)$ torsion-free).
Since $P$ is (almost) $G$-free, there exists a homomorphism $\psi : P \to G$ with $\psi(\phi(\gamma)) \notin V$.
Thus, the map $\psi\circ \phi : \Gamma \to G$ has $\psi \circ \phi (\gamma) \notin V$, as desired.
\end{proof}

The next lemmas demonstrate that $\Lt$, like the class of residually free groups, is closed under direct products and passage to subgroups.

\begin{lemma} \label{DirectProductLemma}
(Finite) direct products of (almost) $G$-free groups are (almost) $G$-free.
\end{lemma}

\begin{proof}
Let $G$ be a linear algebraic group and $V$ an arbitrary subvariety of $G$.
Let $\Gamma = \prod_{i =1}^N \Gamma_i$ be a direct product of (almost) $G$-free groups.
Set $\mathcal{P} = \{ \Gamma_i \}_{i=1}^N$.
Using the natural projections onto $\Gamma_i$, we see that $\Gamma$ is (almost) detectable by $\mathcal{P}$.
Lemma \ref{DetectableLemma} then implies that $\Gamma$ is (almost) $G$-free.
\end{proof}

\begin{lemma} \label{ContainmentLemma}
Subgroups of (almost) $G$-free groups are (almost) $G$-free.
\end{lemma}

\begin{proof}
Let $\Gamma$ be a (almost) $G$-free group and $\Delta \leq \Gamma$.
By using the injection map $\phi : \Delta \to \Gamma$ induced by $\Delta \leq \Gamma$, we see that $\Delta$ is (almost) residually $\{ \Gamma \}$.
As $\Gamma$ is (almost) $G$-free, we have that $\Delta$ is (almost) $G$-free by Lemma \ref{DetectableLemma}.
\end{proof}

The next lemma will be used in \S\ref{Whether}.

\begin{lemma} \label{ProductLemma}
Let $\Gamma$ be a $G_i$-free group for each $i=1,\ldots, k$.
Then $\Gamma$ is $\prod_{i=1}^k G_i$-free.
\end{lemma}

\begin{proof}
Let $G = \prod_i G_i$.  Then
$$\Hom(\Gamma,G) = \prod_i \Hom(\Gamma,G_i),$$
and
$$e_{G,\gamma}(\Hom(\Gamma,G)) = \prod_i e_{G_i,\gamma} (\Hom(\Gamma,G_i)).$$
Given a finite collection of topological spaces $X_i$ and dense subsets $D_i\subset X_i$, by definition of product topology, $\prod_i D_i$ is dense in $\prod_i X_i$.  Applying this to $D_i := e_{G_i,\gamma}$ and $X_i := G_i$, we obtain the lemma.
\end{proof}

\subsection{Connections with linearity and a Tits alternative}
\label{LinearSection}

The next proposition, coupled with the fact that finite groups are never $G$-free (Lemma \ref{TorsionFreeLemma}), shows that 
for finitely generated groups, being $G$-free is a stronger condition than being linear.

\begin{rem} \label{GrigorchukGroupRemark}
It is not true that an almost $G$-free group is necessarily linear.
The first Grigorchuk group is a finitely generated \cite[Corollary VIII.15]{MR1786869}, residually finite \cite[Proposition VIII.6]{MR1786869}, nonlinear \cite[Corollary VIII.19]{MR1786869} group consisting only of torsion elements \cite[Theorem VIII.17]{MR1786869}.
Hence, the first Grigorchuk group is almost $G$-free for any semisimple $G$, but is not linear.  It now follows that the last arrow in (\ref{chain}) cannot be reversed.
\end{rem}

\begin{prop} \label{LinearEmbeddingLemma}
If $\Gamma$ is a finitely generated group that is $G$-free, then the group $\Gamma$ is linear.
If $\,\Gamma$ is a finitely generated group that is almost $G$-free, then $\Gamma$ is an extension of a linear group by a torsion group.
\end{prop}

\begin{proof}
Since $\Gamma$ is finitely generated, its representation variety is a subvariety of $G^{\rank(\Gamma)}$, and thus has finitely many irreducible components (see Appendix \ref{AlgebraicGeometryBackgroundAppendix}).
Let $\Phi$ be the finite collection of irreducible components of $\Hom(\Gamma, G)$.
For each nontorsion $\gamma \in \Gamma$, there exists some $\Omega \in \Phi$ such that $e_{G,\gamma}(\Omega)$ is dense in $G$.
For any $\Omega \in \Phi$, let $S_\Omega$ denote the collection of $\gamma$ for which this density condition holds.  
Since $\Gamma$ is almost $G$-free, $\bigcup_{\Omega \in \Phi} S_\Omega$ consists of all nontorsion elements of $\Gamma$.

By the Baire category theorem (Theorem~\ref{Baire}), the intersection
$$
\bigcap_{\gamma \in S_\Omega} \{\phi\in \Omega\mid \phi(\gamma)\neq 1\}
$$
is always a non-empty subset (note that $\bigcap_{\gamma \in \emptyset} \{\phi\in \Omega\mid \phi(\gamma)\} = \Omega$).
For each $\Omega \in \Phi$, we select a single $\phi_\Omega$ from this non-empty set.
The kernel of the natural homomorphism from $\Gamma$ to 
\begin{equation}
\label{multiG}
\prod_{\Omega \in \Phi} \Gamma / \ker (\phi_\Omega)
\end{equation}
contains only torsion elements since no $\gamma\in S_\Omega$ lies in $\ker(\phi_\Omega)$.
On the other hand, each $\Gamma / \ker (\phi_\Omega)$ can be realized as a subgroup of $G$ via $\phi_\Omega$.
Thus, $\Gamma$ is an extension of a linear group by a torsion group, and if $\Gamma$ is torsion-free, it is linear.
\end{proof}

The next lemma demonstrates that virtually solvable groups that are in $\Lf$ are actually virtually free abelian.
It is possible, however, for a solvable group that is not virtually abelian to be in $\Lt$:

\begin{rem} \label{LamplighterGroupRemark}
Let $\Gamma$ be the Lamplighter group $\Z/2\Z \wr \Z$. Set $\Delta = \op_{i \in \Z} \Z/2\Z$ to be the base group of $\Gamma$ so $\Gamma/\Delta \cong \Z$.
Every element in $\Delta$ is of order $2$. It follows from Lemma \ref{DetectableLemma} that the lamplighter group is almost $G$-free for any semisimple group $G$.
\end{rem}

\begin{lemma} \label{SolvableGroupLemma}
Let $\Gamma$ be a finitely generated group.
If $\Gamma$ is virtually solvable and $G$-free for some semisimple group $G$ then 
it is virtually free abelian.
\end{lemma}

\begin{proof}
The unipotent elements in $G$ form a proper closed subvariety as $G$ is semisimple.
Following the proof of Proposition~\ref{LinearEmbeddingLemma},
there exist homomorphisms $\phi_\Omega\colon \Gamma\to G(\C)$, one for each component of
$\Hom(\Gamma,G)$, such that $\phi_\Omega(\gamma)$ is
not unipotent for all $\gamma\in S_{\Omega}$.
If $Q_\Omega$ denotes the Zariski closure of $\phi_\Omega(\Gamma)$,
then each $Q_\Omega$ is virtually solvable, so each identity component $Q_\Omega^\circ$ is connected solvable and therefore
contained in a Borel subgroup $B_\Omega\subset G$.   If 
$$\Gamma^\circ:= \Gamma\cap \bigcap_\Omega \phi_\Omega^{-1}(Q_\Omega^\circ(\C)),$$
then $\Gamma^\circ$ is of finite index in $\Gamma$.
If $\gamma\in[\Gamma^\circ,\Gamma^\circ]$,
then 
$$\phi_\Omega(\gamma)\in [B_\Omega(\C),B_\Omega(\C)]$$
is unipotent, so $\gamma\not\in S_\Omega$ for all $\Omega$.  It follows that $\gamma=1$, which means
that  $\Gamma$ is virtually abelian.  
\end{proof}

We now present a version of Tits alternative for groups in $\Lf$.
Note that the condition that $\Gamma$ be finitely generated is needed.
For instance, $\Q$ is in $\Lf$ (the image of any element is dense in a maximal torus) but is not virtually free abelian.

\begin{proof}[Proof of Theorem~\ref{ContainsFreeSubgroupTheorem}]
By Tits' alternative and Proposition~\ref{LinearEmbeddingLemma}, we have that $\Gamma$ is virtually solvable or contains a nonabelian free group.
We can therefore assume that $\Gamma$ is virtually solvable.
By Lemma \ref{SolvableGroupLemma}, the group $\Gamma$ is virtually free abelian.
Since $\Gamma$ is torsion-free by Lemma \ref{TorsionFreeLemma}, we are done by Theorem \ref{VirtuallyAbelianNotSLnFree} below.
\end{proof}

\section{Determining when a group satisfies Borel's Theorem}
\label{Whether}

\subsection{Conditions for torsion-free groups}
\label{DeterminingGFreeSection}

In this section, we present several variants of Borel's original proof that free groups are $G$-free for all semisimple groups $G$.

\begin{lemma}
\label{SL-reduction}
Let $\Gamma$ be a finitely generated group.  If $\Gamma$ is $\SL_n$-free for all $n\ge 2$, then $\Gamma$ is $G$-free for all semisimple $G$.
\end{lemma}

\begin{proof}
Let $e_{G,\gamma}\colon \Hom(\Gamma,G)\to G$ denote the evaluation map at $\gamma$.
For any homomorphism $\phi\colon H\to G$, $e_{G,\gamma}$  contains $\phi(e_{H,\gamma}(\Hom,\Gamma,H))$ and is closed under conjugation by $G$.
Therefore, if $e_{H,\gamma}$ has Zariski-dense image  in $H$, and the set of $G$ conjugates of $\phi(H)$ is Zariski-dense in $G$, then
$e_{G,\gamma}(\Hom(\Gamma,G))$ is Zariski-dense in $G$.
Also, if $e_{H_i,\gamma}$ has Zariski-dense image for $i=1,\ldots,n$, then $e_{\prod_i H_i,\gamma}$ has Zariski-dense image (see Lemma \ref{ProductLemma}).  
Every semisimple group admits a surjective homomorphism from a product of simply connected groups which are simple modulo center, and
every semisimple group $G$ which is simple modulo center admits a surjective homomorphism from a product of groups of type $\SL_{n_i}$ which contains a maximal torus of $G$
(and therefore has the property that the union of $G$-conjugates is Zariski-dense).

\end{proof}

\begin{thm}
\label{G-free}
Let $\Gamma$ be a finitely generated group such that
\begin{enumerate}
\item $\Hom(\Gamma,\SL_n)$ is irreducible for all $n\ge 2$.
\item $\Gamma$ is $\SL_2$-free.
\item $\Gamma$ is $\SL_3$-free.
\end{enumerate}
Then $\Gamma$ is $G$-free for all semisimple groups $G$.
\end{thm}

\begin{proof}
Let $e_{n,\gamma}\colon \Hom(\Gamma,\SL_n)\to \SL_n$ denote the evaluation map at $\gamma$.
Let $X_{n,\gamma}$ denote the closure of the image of $e_{n,\gamma}$.  By (1), $X_{n,\gamma}$ is irreducible for all $n\ge 2$ and $\gamma\in\Gamma$.
For $\gamma\neq 1$, by (2) and (3), $X_{2,\gamma} = \SL_2$ and $X_{3,\gamma} = \SL_3$.  We use induction on $n$ to prove $X_{n,\gamma} = \SL_n$ for all $n\ge 2$.

If $n\ge 4$, the obvious embedding $\SL_{n-1}\subset \SL_n$ and the induction hypothesis imply $X_{n,\gamma}$ contains $\SL_{n-1}$, and of course it is invariant under conjugation in $\SL_n$.  The Zariski-closure of the set of all $\SL_n$ conjugates of $\SL_{n-1}$ is the codimension $1$ subvariety of $\SL_n$ consisting of elements for which $1$ is an eigenvalue.  As $X_{n,\gamma}$ is irreducible, it  consists either of this subvariety or of all $\SL_n$.  Applying the induction hypothesis to the embedding $\SL_{n-2}\times \SL_2\subset \SL_n$, we obtain $\SL_{n-2}\times \SL_2\subset X_{\gamma,n}$, which proves $X_{\gamma,n} = \SL_n$.  The theorem now follows from Lemma~\ref{SL-reduction}.

\end{proof}

The conditions of Theorem~\ref{G-free} are in general not easy to check.  For condition (1), we have the following proposition:

\begin{prop}
\label{Count}
Suppose $\Gamma$ is a group with $d$ generators and $r$ relations 
and for all $n\ge 2$ and for each prime $p$ sufficiently large,
$$|\Hom(\Gamma,\SL_n(\F_{p^m}))| = (1+o(1))p^{m(n^2-1)(d-r)}.$$
Then condition (1) holds for $\Gamma$.
\end{prop}

\begin{proof}
By Theorem~\ref{generic-irred}, it suffices to prove that the characteristic $p$ representation scheme 
$\Hom(\Gamma,\SL_{n,\F_p})$ is geometrically irreducible for all $p$ sufficiently large.
By Theorem~\ref{converse-LW} and the estimate for $|\Hom(\Gamma,\SL_{n,\F_p})(\F_{p^m})|$, it follows that there is a unique
geometric component of $\Hom(\Gamma,\SL_{n,\F_p})$ of dimension $(d-r)\dim \SL_n$ and that all other geometric components 
are of lower dimension.  

Let 
$$\Gamma = \langle x_1,\ldots,x_d: R_1,\ldots,R_r\rangle$$
be a presentation of $\Gamma$ with $d$ generators and $r$ relations.
Applying Theorem~\ref{dim-variation} to the multi-word map 
$$(R_1,\ldots,R_r)\colon \SL_{n,\F_p}^d\to \SL_{n,\F_p}^r,$$
the minimum dimension of a geometric component of $\Hom(\Gamma,\SL_{n,\F_p})$  is at least 
$(d-r)\dim \SL_{n,\F_p}$, and we are done.
\end{proof}

The point-counting hypothesis of Proposition~\ref{Count} can be verified for some interesting $1$-relator groups.

\begin{prop}  \label{CompactOrientedSurfacesSatisfy1}
Let $\Gamma = S_g$, the fundamental group of an oriented surface of genus $g \geq 2$.
Then for each prime $p$ sufficiently large,
$$|\Hom(\Gamma,\SL_n(\F_{p^m}))| = (1+o(1))p^{m(n^2-1)(2g-1)}.$$
\end{prop}

\begin{proof}
By a theorem of Frobenius \cite[Proposition 4.1]{MR2369828}, the number of 
ways of representing an element $h$ of a finite group $H$ as $xyx^{-1}y^{-1}$ for $x,y\in H$ is
$$|H|\sum_\chi \frac{\chi(h)}{\chi(1)},$$
where the sum is taken over all irreducible characters $\chi$ of $H$.
By the generalized orthogonality relation
\cite[Th.~2.13]{MR0460423} and induction on $n$, we obtain
$$|\Hom(S_g,H)| = |H|^{2g-1}\sum_\chi \frac{1}{\chi(1)^{2g-2}}.$$
By a result of Martin Liebeck and Aner Shalev \cite[Th.~1.1]{MR2116277}, we have 
$$\lim \sum_{\chi\neq 1} \frac{1}{\chi(1)^{2g-2}} = 0,$$
where the limit is taken over any sequence of groups of the form $G(\F_q)$, where $G$ is simply connected and semisimple
of fixed rank.  This implies the result of Jun Li \cite{MR1206154}
that $\Hom(S_g,G)$ is irreducible not only for $G$ of the form $\SL_{n,\C}$ but for all simply connected
semisimple groups $G$. 
\end{proof}

Likewise, we obtain

\begin{prop} \label{ProjectivePlanesSatisfy1}
If $\Gamma = \langle x_1,\ldots,x_m|x_1^2 x_2^2\cdots x_m^2\rangle,$
then for each prime $p$ sufficiently large,
$$|\Hom(\Gamma,\SL_n(\F_{p^m}))| = (1+o(1))p^{m(n^2-1)(m-1)}.$$
\end{prop}

\begin{proof}
 The formula counting 
homomorphisms from $\Gamma$ to a finite group $H$ is
$$|\Hom(\Gamma,H)| = |H|^{m-1}\sum_\chi \frac{\iota(\chi)^m\chi(g)}{\chi(1)^{m-1}},$$
where $\iota$ is the Frobenius-Schur indicator.  The proof is essentially the same as before, the starting point being the
classical theorem of Frobenius and Schur \cite[Th.~4.5]{MR0460423}, that for any finite group $H$, the number of solutions in $H$ of $x^2=h$ is
$$\sum_\chi \iota(\chi) \chi(h).$$
\end{proof}

\begin{prop} \label{GetCondition2Prop}
In Theorem \ref{G-free}, under assumption (1), we can deduce assumption (2) from the assertion that
there exists a homomorphism $i_2\colon \Gamma\to \SL_2(\C)$ such that $i_2(\gamma)$ is not unipotent for $\gamma\neq 1$.
\end{prop}

\begin{proof}
The condition on $i_2$ is equivalent to $\trace(i_2(\gamma))\neq 2$, so that $\trace\circ e_{2,\gamma}$ is a non-constant function in $\rho\in \Hom(\Gamma,\SL_2)$.
As $\Hom(\Gamma,\SL_2)$ is irreducible, it follows that the image of $\trace\circ e_{2,\gamma}$ is Zariski-dense in the affine line.  Any closed subvariety of $\SL_2$ which
is invariant under conjugation and has a Zariski-dense set of traces is all of $\SL_2$.  Note that if $i_2$ maps $\Gamma$ to $\SU(2)$, it suffices to assume that it is injective.
\end{proof}

If $D$ is a central division algebra of degree $n$ 
over a field $K$, following standard notation \cite[Table II]{MR0224710}, we denote by $\SL_1(D)$ the algebraic group over $K$ 
whose $K$-points give the elements of $D^\times$ of reduced norm $1$, while its $\bar K$-points give $\SL_n(\bar K)$.
We recall that reduced trace gives a map from $D\to K$, and applying reduced trace to all integer powers of an element of $D$, we see that
the power sums of the eigenvalues of any element of $\SL_1(D)\subset \SL_n(\bar K)$ lie in $K$.  If $K$ is of characteristic zero, this implies that
the characteristic polynomial of every element of $\SL_1(D)$ has coefficients in $K$.  
If any element of $\lambda\in K$ is an eigenvalue of $\alpha\in \SL_1(D)$, then $\alpha-\lambda$ is not invertible, so it is zero, and $\alpha = \lambda$
lies in the center of $\SL_1(D)$.
If $D$ is of degree $3$ and the characteristic polynomial of
$\alpha\in \SL_1(D)$
has a multiple root $r$, then $r\in K$.  Thus every element of $\SL_1(D)$ is central or regular semisimple.

Using this observation, we can replace (3)
in Theorem~\ref{G-free} as follows:

\begin{prop}
\label{Div-Alg}
Under hypotheses (1) and (2) of Theorem~\ref{G-free}, if there exists a degree 3 division algebra $D\subset M_3(\C)$ such that
$\Hom(\Gamma,\SL_1(D))\subset \Hom(\Gamma,\SL_3(\C))$ is Zariski-dense in $\Hom(\Gamma,\SL_3)$, then (3) follows and therefore
$\Gamma$ is $G$-free for all semisimple groups $G$.
\end{prop}

\begin{proof}
Let $\gamma\in\Gamma$ be a non-trivial element.  As $X_{2,\gamma} = \SL_2$, we see that $X_{3,\gamma}$ contains the codimension $1$ subvariety of 
$\SL_3$ consisting of matrices for which $1$ is an eigenvalue.  It suffices to prove that there is at least one point of $X_{3,\gamma}$ for which $1$ is not an eigenvalue.  However, for $\rho\in \Hom(\Gamma,\SL_1(D))\subset \Hom(\Gamma,\SL_3(\C))$, $1$ can be an eigenvalue if and only if $\rho(\gamma) = 1$.
If $\rho(\gamma) = 1$ for all $\rho\in \Hom(\Gamma,\SL_1(D))$, Zariski-density implies the same for all 
$\rho\in \Hom(\Gamma,\SL_1(D))$, contrary to the non-triviality of $X_{3,\gamma}$.
\end{proof}

A related criterion for (3) is the following:

\begin{prop}
\label{Local}
Under hypotheses (1) and (2) of Theorem~\ref{G-free}, if there exists a degree 3 division algebra $D$ over a characteristic zero local field $K$
such that $\Hom(\Gamma,\SL_1(D))$ contains a regular point of $\Hom(\Gamma,\SL_3(\bar K))$, then (3) follows and therefore
$\Gamma$ is $G$-free for all semisimple groups $G$.
\end{prop}

\begin{proof}
Applying Theorem~\ref{density} to $X = X_{3,\gamma}$ and
identifying $\bar K$ and $\C$ by the axiom of choice, the proposition now follows from Proposition~\ref{Div-Alg}.
\end{proof}

\begin{prop} \label{Condition3}
Let $\pi_1$ be the fundamental group of the connected sum of three projective planes.
Assuming hypotheses (1) and (2) in Theorem \ref{G-free} hold for $\pi_1$, hypothesis (3) holds as well.
\end{prop}

\begin{proof}
To apply Proposition~\ref{Local}, we observe that
the trivial representation $\pi_1\to \SL_3(\C)$ is a nonsingular point of $\Hom(\pi_1,\SL_3)$.
Indeed, identifying $\Hom(\Gamma,\SL_3)$ with 
$$e_{\SL_3,x_1^2x_2^2x_3^2}^{-1}(1),$$
it suffices by Theorem~\ref{smooth-vs-dominant} to note that the morphism $\SL_3^3\to \SL_3$ given by the word $x_1^2 x_2^2 x_3^2$ induces a
surjective map on tangent spaces at  $(1,1,1)$.
The induced map on tangent spaces $\sl_3^3\to \sl_3$ sends $(X_1,X_2,X_3)$ to $2X_1+2X_2+2X_3$ and is therefore surjective.
\end{proof}

We can now apply the previous results to obtain a new class of groups in $\Lf$.
The proof presented below, specifically in the three projective plane case, can be modified to give a new proof that oriented surface groups are in $\Lf$.

\begin{prop} \label{ThreeProjectivePlanesProp}
Let $S$ be the connected sum of three or more projective planes.
Then $\pi_1(S, \cdot)$ is in $\Lf$.
\end{prop}

\begin{proof}
Let $\pi_1 = \pi_1(S, \cdot)$.
We break the proof into two cases, depending on $k$.
\begin{enumerate}
\item If $k = 3$, then $\pi_1 = \left< x, y, x^2 y^2 z^2 \right>$.
This group is known not to be residually free \cite{MR0162838}.
By Theorem \ref{G-free} and Propositions \ref{Count}, \ref{ProjectivePlanesSatisfy1}, \ref{Local}, and \ref{Condition3} it suffices to show that there exists some map $\phi: \pi_1 \to \SL_2(\C)$ such that $\phi(\pi_1)$ does not contain any unipotent elements.
To do this, identify $\pi_1$ with $P$, the cocompact subgroup of isometries of the hyperbolic plane which is the universal covering space 
of the connected sum of three real projective planes. 
Let $\pi$ denote the homomorphism from 
$$\SL_2(\C) \cap \langle i\rangle \GL_2(\R)$$
to the group of all M\"obius transformations given by Sepp\"al\"a and Sorvali \cite[\S6]{MR1241817}.  Then
there exists a lift $\tilde P$ of $P$ so that $\pi$ gives an isomorphism $\tilde P\to P$  \cite[Theorem 6]{MR1241817}).
We claim that $\tilde P$ has no nontrivial unipotent elements. 
Indeed, $P$ contains, as a subgroup of index 2, a discrete and cocompact subgroup of $\PSL(2, \R)$ (the orientation-preserving isometries of the hyperbolic plane). 
Let $X$ be a nontrivial unipotent element in $\tilde P$. Then $\pi(X^2) = \pi(X)^2$ is unipotent and lies inside $\PSL(2, \R)$ and so $\pi(X^2)$ is parabolic. This is impossible, as no discrete and cocompact subgroup of $\PSL(2, \R)$ contains a parabolic element (\cite[Theorem 4.2.1]{MR1177168}). It follows that $\tilde P$ does not contain any nontrivial unipotent elements, as desired.

\item If $k > 3$, then by \cite{MR0215903}, $\pi_1$ is residually free. Thus, by Lemma \ref{DetectableLemma}, $\pi_1 \in \Lf$.
\end{enumerate}
\end{proof}

\subsection{Conditions for general groups}
\label{weakGSection}
\newcommand{\y}{\mathbf{y}}
\newcommand{\z}{\mathbf{z}}

In this section, we present some conditions for a finitely generated group $\Gamma$ to be almost $G$-free for all semisimple groups $G$.
Our main theorem is a variant of the results in the previous section.  This variation is forced on us. We cannot expect that $\Hom(\Gamma,\SL_n)$ will be connected
if $\Gamma$ has non-trivial torsion elements. Indeed, there are typically several different conjugacy classes of elements of $\SL_n(\C)$ of given order $m>1$.  We therefore try to pin down the class of the image of each torsion conjugacy class.  We assume that
$\Gamma$ has finitely many classes of non-trivial elements of finite order, and we denote by $x_1,\ldots,x_k$ representatives of each class.

If $G$ is a semisimple group defined over $\C$ and $y\in G(\C)$ is of finite order, it is semisimple, and its conjugacy class is therefore
closed.  Since $G$ is irreducible, its conjugacy classes are likewise irreducible.  If $\y = (y_1,\ldots,y_m)$ is an $m$-tuple of semisimple
elements of $G(\C)$, we denote by $V(G,\y)$ the closed subvariety
$$V(G,\y) := e_{x_1,\ldots,x_k}^{-1}(C_1\times\cdots\times C_k)\subset \Hom(\Gamma,G),$$ 
where $C_i$ is the conjugacy class of $y_i$, and 
$$e_{x_1,\ldots,x_k}\colon \Hom(\Gamma,G)\to G^k$$
is the multiword evaluation map.
If $G = \SL_n$, we denote $V(\SL_n,\y)$ by $V_n(\y)$ for brevity.

\begin{thm} \label{GeneralToolTheorem}
Let $\Gamma$ be a finitely generated group with finitely many conjugacy classes of non-trivial elements of finite order,
represented by elements $x_1,\ldots,x_k$.
For each $n\ge 2$, let $Y_n\subset \SL_n(\C)^k$ be a non-empty set  of $k$-tuples of semisimple elements.
Suppose:
\begin{enumerate}
\item For each $n$ and $\y\in Y_n$, the variety $V_n(\y)$ is irreducible.
\item For each $n\ge 3$ and each $\y\in Y_n$, 
there exists $\y'\in Y_{n-1}$ and a 1-dimensional character $\chi$ of $\Gamma$ such that
$\chi(x_i)y'_i\oplus \chi(x_i)^{1-n}$ is conjugate
in $\SL_n(\C)$ to $y_i$ for $i=1,2,\ldots,k$.
\item For each $n\ge 4$ and each $\y\in Y_n$, there exists $\y^1\in Y_{n-2}$, $\y^2\in Y_2$ and 
1-dimensional characters $\chi_1,\chi_2$ of $\Gamma$ such that $\chi_1(x_i)y^1_i\oplus \chi_2(x_i) y^2_i$ is conjugate in $\SL_n(\C)$
to $y_i$ for $i=1,2,\ldots,k$.
\item For each $\y\in Y_2$, there exists an injective homomorphism $\Gamma\to \SL_2(\C)$ in $V_2(\y)$
such that $\rho(\Gamma)$ contains no non-trivial unipotent element.
\item For each $\y\in Y_3$, there exists a regular point in $V_3(\y)$
corresponding to a homomorphism $\Gamma\to \SL_3(\bar K)$ whose image lies in $\SL_1(D)$ for some degree $3$
division algebra $D$ over a $\ell$-adic field $K$.
\end{enumerate}

Then $\Gamma$ is almost $G$-free for all semisimple $G$.
\end{thm}

\begin{proof}

For any $\gamma \in \Gamma$, the evaluation map $e_{G,\gamma}$ restricts to a map $V(G,\y)\to G$ which we  denote $e_{G,\y,\gamma}$.  For any homomorphism $\phi\colon G\to H$ we have commutative a diagram
$$\xymatrix{V(G,\y) \ar[r]^\phi \ar[d]_{e_{G,\y,\gamma}}&V(H,\phi(\y)) \ar[d]^{e_{H,\phi(\y),\gamma}}\\
G \ar[r]^\phi & H},$$
so the closure $X_{H,\phi(\y),\gamma}$ of the image  of $e_{H,\phi(\y),\gamma}$ contains $\phi(X_{G,\y,\gamma})$.  Also, $V(G,\y)$ depends only on the conjugacy classes of the $y_i$ and therefore admits a conjugacy action by $G$ which the evaluation maps respect.  Thus $X_{G,\y,\gamma}$ is a closed, conjugation-invariant subvariety of $G$.  It follows that if for all $n\ge 2$ there exists $\y\in Y_n$ such 
that $X_{\SL_n,\y,\gamma} = \SL_n$, then $e_{G,\gamma}(\Hom(\Gamma,G))$ is dense in $G$ for all semisimple $G$.

For $n\ge 2$ and $\y\in Y_{n+1}$, there exists $\y'$ and $\chi$ such that
the homomorphisms 
$$\Hom(\Gamma,\SL_n)\to \Hom(\Gamma,\SL_{n+1})$$
defined by 
$$\rho_n\mapsto \rho_n\otimes \chi\oplus \chi^{\otimes -n}$$
and condition (2) guarantees that $V_n(\y')$ maps to $V_{n+1}(\y)$.  Thus, we have commutative diagrams
$$\xymatrix{V_n(\y') \ar[r]\ar[d]_{e_{n,\gamma}} &V_{n+1}(\y) \ar[d]^{e_{n+1,\gamma}} \\
\SL_n\ar[r] &\SL_{n+1}}$$
where the bottom row sends 
$$M\mapsto \chi(\gamma)M\oplus \chi(\gamma)^{-n}.$$
Let $S$ be the union of all conjugacy classes of the image of this function.
Note that any element in $S$ has at least one eigenvalue which has order $n$.

For each $\y\in Y_{n+2}$ we have homomorphisms 
$$\Hom(\Gamma,\SL_n)\times \Hom(\Gamma,\SL_2)\to \Hom(\Gamma,\SL_{n+2})$$
defined by
$$(\rho_n,\rho_2)\mapsto \rho_n\otimes \chi_1\oplus \rho_2\otimes \chi_2$$
which map $V_n(\y^1)\times V_2(\y^2)\to V_{n+2}(\y)$, and there is a commutative diagram
$$\xymatrix{V_n(\y^1)\times V_2(\y^2) \ar[r]\ar[d]_{e_{n,\gamma}\times e_{2,\gamma}} &V_{n+2}(\y) \ar[d]^{e_{n+2,\gamma}} \\
\SL_n\times \SL_2\ar[r] &\SL_{n+2},}$$
where the bottom row sends
$$M_1\times M_2\mapsto \chi_1(\gamma)M_1\oplus \chi_2(\gamma)M_2.$$
By conditions (4) and (5), the image of this map contains an element with no eigenvalues being roots of unity.
It follows that this image contains an element which is not in $S$.
Thus, we can use induction on $n$ to prove that the $e_{n,\gamma}$ all have Zariski-dense image, provided we can treat
the base cases $n=2$ and $n=3$.

For $n=2$, we use (4) together with the fact that an irreducible closed subvariety of $\SL_2$ which is a union of conjugacy classes 
and contains both $1$ and a non-unipotent elements is all of $\SL_2$.  For $n=3$, we use the fact that a conjugation-invariant closed
irreducible subvariety of $\SL_3$ which contains $\SL_2$ and some element without eigenvalue $1$ is all of $\SL_3$.  
Although the particular homomorphism $\rho_3$ whose existence is guaranteed by (5) might have a nontorsion element $\gamma$
in its kernel, the homomorphisms in an $\ell$-adic neighborhood of $\rho_3$ cannot be identitically trivial on $\gamma$.  Indeed,  they
are Zariski-dense in $V_3(\y)$ and $V_3(\y)$ contains at least one injective representation, namely the representation coming via Condition (1) from 
the injective $\SL_2$-representations of $\Gamma$ guaranteed by Condition (4).
\end{proof}

We can now prove Theorem~\ref{SevenSevenTheorem}:
if $\ell$ is a prime which is $1$ (mod $3$), then $\Gamma := \Z/\ell\Z\ast\Z/\ell\Z$ is almost $G$-free for all semisimple $G$.

\begin{proof}[Proof of Theorem~\ref{SevenSevenTheorem}]
Let $\gamma_1$ and $\gamma_2$ denote generators of the two free factors $\Z/\ell\Z$.  
By  \cite[Cor.~4.1.4]{MR2109550} and \cite[Cor.~4.1.5]{MR2109550}, there are $2(\ell-1)$ different conjugacy classes of non-trivial elements of
$\Gamma$ of finite order, and they are represented by  
$$x_1=\gamma_1,\,x_2 = \gamma_1^2,\,\ldots,\,x_{\ell-1}=\gamma_1^{\ell-1},\,x_\ell = \gamma_2,\,\ldots,\,x_{2\ell-2}=\gamma_2^{\ell-1}.$$
All of our $\y$ will be of the form 
$$(y_1,y_1^2,\ldots,y_1^{\ell-1},y_2,y_2^2,\ldots,y_2^{\ell-1}),$$
where $y_1,y_2\in \SL_n(\C)$ are
of order $\ell$, so $V_n(\y) = C_1\times C_2$ , where $C_1$ and $C_2$ denote the conjugacy classes of $y_1$ and $y_2$ respectively.
This implies Condition (1) of Theorem \ref{GeneralToolTheorem}.

In order to define $Y_n$ precisely, we first define for each integer $k\in [0,\ell-1]$ a set $B_k$ of subsets $S\subset \F_\ell$.
Since $\ell \equiv 1 ($mod $3)$, there exists a unique $3$-element subgroup $\mu_3\subset \F_\ell^\times\subset \F_\ell$.
We let $B_k$ consist of all $S \subset \F_\ell$ of cardinality $k$ which sum to $0$ and satisfy the additional condition for $k\ge 3$ that
$S$ contains the image under some affine transformation of $\mu_3$.  For any integer $n\ge 2$, we define
$A_n$ to be the set of all functions 
$$f\colon \F_\ell\to \{\lfloor n/\ell\rfloor,\lceil n/\ell\rceil\}$$ 
such that
$$\{x\in \F_\ell\mid f(x) > n/\ell\}\in B_k,$$
where $k\in [0,\ell-1]$ is the (mod $\ell$) reduction of $n$.

We fix an injective homomorphism $\psi$ from $(\F_\ell,+)$ to $\C^\times$ and
to any $f\in A_n$, we associate the conjugacy class $C_f\subset \SL_n(\C)$ consisting of $\ell$th roots of the identity
for which the eigenvalue $\psi(x)$ occurs with multiplicity $f(x)$.  
For each conjugacy class $C_f$, we select any element $\y^f$ for which $y_1=y_2$ belongs to $C_f$. We let 
$$Y_n = \{\y^f\mid f\in A_n\}.$$

Since there exists a character which takes the value $\psi(1)$ on both $\gamma_1$ and $\gamma_2$, to prove (2) it suffices to show that every element of $A_n$ is a sum of a translate of an element of $A_{n-1}$ and the translate of an element of $A_1$.
When $\ell\nmid n$, it suffices 
to prove that for $1\le k\le \ell-1$,
every element of $B_k$ is the union of a single element of $\F_\ell$ and an additive translate of an element of $B_{k-1}$.
Clearly, every $k$-element subset of $\F_\ell$ has a translate which sums to zero,
which proves the claim.  To finish (2), we note that when $\ell\mid n$, $A_n$ consists of the single element
$(n/\ell,n/\ell,\ldots,n/\ell)$ which decomposes as a sum of an element of $A_{n-1}$ and an element of $A_1$:
$$(n/\ell,n/\ell,\ldots,n/\ell) = (n/\ell-1,n/\ell,\ldots,n/\ell) + (1,0,\ldots,0).$$

Likewise, we can prove (3) when $n$ reduces to $k\ge 2$ (mod $\ell$) by showing that every element of $B_k$ is a union of a translate of an element of $B_{k-2}$ and a translate of an element of $B_2$.  As every $2$-element set is a translate of an element of $B_2$, this is clear.  So we must deal with two cases: $k=0$ and $k=1$.  In these two cases, $A_n$ has only one element, and we use the decompositions
$$(n/\ell,n/\ell,\ldots,n/\ell) = (n/\ell-1,n/\ell,\ldots,n/\ell,n/\ell-1) + (1,0,\ldots,0,1).$$
and
$$(n/\ell+1,n/\ell,\ldots,n/\ell) = (n/\ell,n/\ell,\ldots,n/\ell,n/\ell-1) + (1,0,\ldots,0,1).$$

For Condition (4), it suffices to prove that for any primitive $\ell$th root of unity $\zeta_\ell$, there exists
an injective homomorphism from $\Gamma$ to $\SL_2(\C)$ sending $\gamma_1$ and $\gamma_2$ to matrices with eigenvalues
$\zeta_\ell^{\pm 1}$ and with no non-trivial unipotents in the image.  If we realize $\Gamma$ as a Fuchsian group of the second kind with signature $(1;p,p)$, we achieve such an embedding in $\SL_2(\R)$ for $\zeta_p = e^{2\pi i/p}$, and all other cases can be achieved by composing the resulting homomorphism $\Gamma\hookrightarrow \SL_2(\C)$ with a suitable automorphism of $\C$.

For Condition (5), $V_3(\y)$ is non-singular, so it is just a matter of showing that some homomorphism $\Gamma\to \SL_3(\C)$ in
$V_3(\y)$ with image contained in a suitable $\SL_1(D) \leq \SL_3(\C)$.  As $\Gal(\Q(\zeta_\ell)/\Q) \cong \F_\ell^\times$, there exists an intermediate field $E := \Q(\zeta_\ell)^{\mu_3}$ such that $[\Q(\zeta_\ell):E] = 3$.  
A rational prime $p$ splits completely in $E$ if and only if $p$ reduces (mod $\ell$) to 
an element of $\mu_3$; it splits completely in $\Q(\zeta_\ell)$ if and only if it reduces (mod $\ell$) to $1$.
By Dirichlet's theorem, there exists a prime $p$ which splits in $E$ but not in $\Q(\zeta_\ell)$.  It follows that
$E\subset \Q_p$ but $\zeta_\ell$ is algebraic of degree $3$ over $\Q_p$. 
The Brauer group of $\Q_p$ is canonically isomorphic to $\Z/\Q$ \cite[XIII~Prop.~6]{MR554237},
and we define $D$ to be the (degree $3$) division algebra over $\Q_p$ with invariant $1/3$.
Every degree $3$ extension of $\Q_p$ can be embedded in $D$ \cite[XIII~Prop.~7]{MR554237}.
In particular, there exists an injective $\Q_p$-homomorphism $i\colon \Q_p(\zeta_\ell)\to D$, and it follows that
$$D\otimes_{\Q_p} \Q_p(\zeta_\ell)\cong M_3(\Q_p(\zeta_\ell)).$$
If $\alpha\in \Q_p(\zeta_\ell)$ has minimal polynomial $P(x)$ over $\Q_p$, then
$P(i(\alpha)) = 0$, but viewed as an element of $M_3(\Q_p(\zeta_\ell))$, $i(\alpha)$ has a characteristic polynomial with
coefficients in $\Q_p$, which must then be $P(x)$ as well.  It follows that the eigenvalues of $i(\alpha)\in M_3(\bar\Q_p)$
are $\alpha$ and its conjugates over $\Q_p$.  In particular, if $\zeta_\ell$ is a primitive $\ell$th root of unity, its conjugates
over $E$ (and therefore over $\Q_p$) are $\zeta_\ell^a$ and $\zeta_\ell^{a^2}$, where the image of $a$ in $\F_p$ generates $\mu_3$.

Any element $S\in B_3$ is a coset of the order $3$ subgroup $\mu_3\subset \F_\ell^\times$.
Identifying $\bar\Q_p$ and $\C$, $\psi(S)$ is therefore a Galois-orbit of an element of $\Q_p(\zeta_l)$,
and it follows that there exists an element $e_S\in D\subset M_3(\C)$ with eigenvalues $\psi(S)$.
As $\sum_{s\in S} s = 0$, 
$e_S\in \SL_1(D)$.  The homomorphism  sending $\gamma_1$ and $\gamma_2$ to $e_S$ is 
therefore of the desired kind.
\end{proof}

The same strategy can be used to prove Theorem~\ref{QuadTheorem}:
if $\ell\ge 19$ is a prime that is $\equiv 1$ (mod $3$), then
$$\Gamma := \left< x, y, z, t : x^\ell = y^\ell = z^\ell = t^\ell = xyzt \right>$$
is in $\Lt$.

\begin{proof}[Proof of Theorem~\ref{QuadTheorem}]
Let $x$, $y$, $z$, and $t$ be as in the presentation of $\Gamma$.
By  \cite[Cor.~4.4.5]{MR2109550} and \cite[Th.~4.5]{MR2109550}, there are $4(\ell-1)$ different conjugacy classes of non-trivial elements of
$\Gamma$ of finite order, and they are represented by  
$$x, x^2, \ldots, x^{\ell-1}, y, y^2, \ldots, y^{\ell-1}, \ldots, z, z^2, \ldots, z^{\ell-1}, t, t^2, \ldots, t^{\ell-1}.$$
All of our $\y$ will be of the form 
$$(y_1,y_1^2,\ldots,y_1^{\ell-1},y_1^{\ell-1},y_1^{\ell-2},\ldots,y_1, y_1,y_1^2,\ldots,y_1^{\ell-1},y_1^{\ell-1},y_1^{\ell-2},\ldots,y_1),$$
where $y_1,y_2\in \SL_n(\C)$ are of order $\ell$.  Thus,  $V_n(\y)$ is a subvariety of $\SL_n(\C)^4$ of the form
\begin{equation} \label{SubvarietyEquation}
W:=\{ (X,Y,Z,T) \in C_1 \times C_2 \times C_3 \times C_4 : X Y Z T = 1 \},
\end{equation}
where $C_i$ are specified conjugacy classes of semisimple elements in $\SL_n(\C)$ which satisfy the
condition $C_1 = C_2^{-1} = C_3 = C_4^{-1}$.
Before showing Condition (1) of Theorem \ref{GeneralToolTheorem}, we need to define $Y_n$ precisely.

%
We define $B_k$, $A_n$, $\psi$, and $C_f$ exactly as in the proof of Theorem~\ref{SevenSevenTheorem}.
For each $C_f$, we select any element $\y^f$ for which $y_1=y_2^{-1}$ belongs to $C_f$. We let 
$$Y_n = \{\y^f\mid f\in A_n\}.$$
Thus, the conjugacy classes appearing in $W$ corresponding to $V_n(\y)$ have semisimple elements with multiplicity at most $\lceil n/\ell \rceil.$

We now show Condition (1) of Theorem \ref{GeneralToolTheorem} by showing that varieties of the form (\ref{SubvarietyEquation}) are geometrically irreducible.
By \cite[Th. 7.2.1]{S92} and the fact  that two semisimple elements in $\SL_n(\F_p)$ are conjugate if and only if they are conjugate in $\GL_n(\F_p)$, the number of elements in $W(\F_p)$ is 
\begin{equation} \label{NumberOfSolutions}
\frac{1}{|\GL_n(\F_p)|} |C_1| \cdots |C_4| \sum_{\chi} \frac{\chi(x_1) \cdots \chi(x_4)}{\chi(1)^{2}}
= \frac{1}{|\GL_n(\F_p)|} |C_1|^4 \sum_{\chi} \frac{|\chi(x_1)|^4}{\chi(1)^{2}},
\end{equation}
where $x_i$ is a representative of the conjugacy class $C_i$ in $\GL_n(\F_p)$ and $\chi$ runs through all irreducible characters of $\GL_n(\F_p)$.
As there are exactly $q-1$ characters of $\GL_n(\F_q)$ of degree $1$,
namely those characters which factor through the determinant map $\GL_n(\F_q)\to \F_q^\times$, it follows that
 (\ref{NumberOfSolutions}) is given by
$$
\frac{|C_1|^4}{|\GL_n(\F_p)|}  \left( q-1 +   \sum_{\chi(1) >  1} \frac{|\chi(x_1)|^4}{\chi(1)^{2}} \right).
$$
Note that $C_1$ is the conjugacy class of a semisimple element in a simply connected semisimple group, so by Steinberg's theorem  \cite[Th.~2.11]{MR1343976}, 
it is the quotient of $\SL_n(\F_q)$ by the group of $\F_q$-points of a geometrically connected group over $\F_q$, so 
$$|C_1| = q^{\dim \SL_n - \dim C_{\SL_n}(x_1)}(1+o_q(1)).$$
In the special case that $x_1$ is regular, this is $q^{n^2-n}(1+o_q(1))$

Let $P_q$ denote the set of ordered pairs $(\chi,\chi')$ consisting of an irreducible character $\chi$ of $\SL_n(\F_q)$ and 
an irreducible character $\chi'$ of $\GL_n(\F_q)$ such that $\chi$ is an irreducible constituent of the restriction of $\chi'$ or (equivalently, by Frobenius reciprocity),
$\chi'$ is a constituent of the induced character of $\chi$.  Thus, $P_q$ projects onto the set of irreducible representations of $\SL_n(\F_q)$
and likewise onto the set of irreducible representations of $\GL_n(\F_q)$.  For $(\chi,\chi')\in P_q$, we have
$$\chi(1) \le \chi'(1) \le [\GL_n(\F_q):\SL_n(\F_q)] \chi(1) = (q-1)\chi(1).$$
In particular, the number of characters $\chi'$ associated to a single $\chi$ is at most $q-1$.
The characters of $\GL_n(\F_q)$
associated to the trivial character of $\SL_n(\F_q)$ are precisely the $q-1$ characters of degree $1$.  All other characters of $\GL_n(\F_q)$ have degree 
at least $(q^{n-1}-1)/2$ by the bound of Vicente Landazuri and Gary Seitz
for degrees of non-trivial projective characters of $\mathrm{PSL}_n(\F_q)$ \cite{MR0360852}.
The total number of characters of $\GL_n(\F_q)$ is $O(q^n)$ by a result of Martin Liebeck and L\'aszl\'o Pyber \cite{MR1489911}.  If $x_1$ is regular semisimple, then
$|\chi(x_1)|$ is bounded above by a constant depending only on $n$ \cite{MR2920887}.  Thus,
 (\ref{NumberOfSolutions}) is given by
 $$q^{3n^2-4n+1}(1+o_q(1)).$$

We would like to achieve a similar upper bound when $x_1$ has an eigenvalue with multiplicity greater than one.  Let $\alpha = 1/10$.
If $n$ is divisible by $\ell \leq 19$, each eigenvalue has multiplicity $n/\ell < \alpha n$.  Otherwise, writing
$n = a\ell+k$, $1\le k < \ell$, we have $a\ge 1$, so each eigenvalue has multiplicity 
$$a+1 \le \frac{(a+1)n}{a\ell+1} \le \frac{2n}{\ell+1} \le \alpha n.$$
Let $\beta = 4/9$.  By Theorem \ref{CharacterBoundTheorem}, we have $|\chi(x_1)| \leq \chi(1)^{\beta}$ for all $p$ sufficiently large
and any irreducible character $\chi$ of $\SL_n(\F_p)$.

Following \cite{MR2107038}, we write
$$
\zeta^H(s) = \sum_{\chi} \chi(1)^{-s},
$$
where the sum is taken over irreducible representations of $H$.
Thus, by \cite[Th. 1.1]{MR2107038}, we have
$$\zeta^{\SL_n(\F_p)}(s) = 1 + o_p(1)$$
if $s > 2/n$.  It follows that
$$\frac{\zeta^{\GL_n(\F_p)}(s)}{q-1}  - 1 \le \sum_{\chi(1)\neq 1} \frac{\sum_{\{\chi'\mid (\chi,\chi')\in P_p \}} \chi'(1)^{-s}}{p-1}
\le \zeta_{\SL_n(\F_p)}(s) - 1 = o_p(1).$$
Thus,  
\begin{equation}
\label{W-size}
|W(\F_p)| = \frac{1}{|\SL_n(\F_p)|} |C_1|^4 (1 + o_p(1))  =  p^{3\dim \SL_n - 4\dim C_{\SL_n}(x_1)} (1+o_p(1)).
\end{equation}
By Theorem~\ref{dim-variation}, we have that every geometric component of $W$ has dimension at least
$$
3\dim \SL_n - 4\dim C_{\SL_n}(x_1).
$$
Coupling this with (\ref{W-size}) and with Theorem~\ref{converse-LW}, we deduce that $W/\F_p$ is geometrically irreducible, as desired.

The proofs that Conditions (2), (3), and (4) are satisfied exactly parallel those in the proof of Theorem \ref{SevenSevenTheorem}.
For Condition (5), we need an additional argument to verify that $V_3(\y)$ has a regular point. 
The last result in \cite{MR0169956} gives a sufficient condition for a homomorphism from an oriented Fuchsian group $\Gamma$ to an algebraic group $G$
in characteristic zero to be regular point of $\Hom(\Gamma,G)$; it suffices that the space of coinvariants of the adjoint action of $\Gamma$ on the Lie algebra $\g$ of $G$
is zero.  Equivalently, it suffices that the space of invariants of the coadjoint representation is zero, and if $G$ is semisimple, the adjoint and coadjoint representations are isomorphic, so it suffices that the centralizer of the image of $\Gamma$ in $G$ is zero-dimensional.

Identify $\overline{\Q}_p$ with $\C$ as in the end of the proof of Theorem \ref{SevenSevenTheorem}.
We now finish the proof by showing that we may find a (noninjective) homomorphism $\Gamma \to \SL_3(\C)$ such that the centralizer of the image is zero-dimensional.
Fixing  $\y\in Y_3$ fixes a regular semisimple conjugacy class $C_1$ in $\SL_1(D)$.  We will choose a homomorphism $\gamma_{s,t}: \Gamma\to \SL_1(D)$ in $V(\y)$
defined by 
$$(x,y,z,t)\mapsto (s,s^{-1},t,t^{-1}),$$ 
for $s,t\in C_1$:
Set $s$ to be any element in $C_1$ and let $S$ denote the unique maximal torus in $\SL_3$ containing
$s$.  By Zariski-density of $\SL_1(D)$ in $\SL_3$, there exists  $g\in \SL_1(D)$
that does not lie in the normalizer of $S$ (a proper subvariety of $\SL_3$).   Set $t = g s g^{-1}$.
We claim that the centralizer of $\gamma_{s,t}(\Gamma)$ in $\SL_1(D)$ is contained in the center of $\SL_3(\C)$.
Let $z \in \SL_1(D)$ be an element in this centralizer.
Then because $z$ commutes with $s$, we have $z \in S$.
Further, $z$ commutes with $t$, so $g^{-1} zg$ commutes with $s$ and therefore lies in $S$.
Suppose that $z$ is non-central and in $\SL_1(D)$. Then $z$ is regular (c.f. the observation before Proposition \ref{Div-Alg}) so it belongs to a unique maximal torus, which must be $S$.
Thus, $g$  normalizes $S$, which is impossible by our choice of $g$.  Hence, the only elements in 
$\SL_1(D)$ that commute with every element in $\gamma_{s,t}(\Gamma)$ are in the center of $\SL_3(\C)$.
Since the centralizer of $\gamma_{s,t}(\Gamma)$ in $\SL_1(D)$ is finite, it follows that the centralizer of $\gamma_{s,t}(\Gamma)$ in $\SL_3(\C)$ is zero-dimensional, as desired.
\end{proof}

We finish the section by showing there exists $G$-free groups that are not necessarily residually free and are a semidirect product of free groups.

\begin{thm} \label{GFreeButNotRFTheorem}
The group
$$
\left< a_1, \ldots, a_7, b : b a_1 b^{-1} = a_{2},\; ba_2 b^{-1} = a_{3},\ldots, \;ba_6 b^{-1} = a_{7}, \;ba_7 b^{-1} = a_{1} \right>
$$
is in $\Lf$ but is not residually free.
\end{thm}

\begin{proof}
By construction, $\Gamma$ is a semidirect product $F_7$ with $\Z$. It  maps onto the group
$$\Delta=\left< a, b_1, b_2, \ldots,  b_7  : a b_1 a^{-1} = b_{2}, .., a b_6 a^{-1} = b_7, a b_7 a^{-1} = b_1, a^7 = 1 \right>,$$
which is contained inside $\Z/7 \ast \Z/7$ and is almost $G$-free by Theorem \ref{SevenSevenTheorem}.
The entire set $F_7^\bullet$ is almost detected by $\Delta$.
Any  element of $\Gamma\setminus F_7$ is detected by $\Z$.
Since $\Delta$ and $\Z$ are both $G$-free it follows that $\Gamma$ is so by Lemma \ref{DetectableLemma}. 

On the other hand, $\Gamma$ is not residually free:
Since free groups are residually 2-finite, any residually free group must be residually 2-finite. However, the element $b_1 b_2^{-1}$ must vanish in any 2-group quotient of $\Gamma$.
Indeed, the action of $a$ on $F_7$ has order 7, and so if $a^{2^k} = 1$ for any $k$, we must have that the action of $a$ on $F_7$ is trivial.
\end{proof}

\subsection{Groups that fail to satisfy Borel's Theorem}
\label{NotAlmostGFreeGroupsSection}

Here, we give examples of groups which are not almost $G$-free for some semisimple groups $G$.
We start by showing that it is not always true that free products of $G$-free groups are $G$-free.

\begin{prop} \label{ProductofGFreeThm}
The group $\Gamma := (F_2 \times \Z ) * \Z$ is not residually $\{ \SL_2(\C) \}$ but is the free product of two groups that are in $\Lf$.
\end{prop}

\begin{proof}
The groups $\Z$ and $F_2 \times \Z$ are $G$-free for any semisimple $G$ by Borel's Theorem \cite{B83} and Lemma \ref{DirectProductLemma}.

We now show that $\Gamma$, which 
has presentation 
$$\Gamma = \left< a_1, a_2, b, c : [a_1, b] = 1, [a_2, b] = 1 \right>,$$
is not residually $\{\SL_2(\C)\}$.
Supposing, for the sake of contradiction, that the element
$$[[a_1, a_2], [b,c]]$$ 
does not vanish in the image of some homomorphism, $\phi: \Gamma \to \SL_2(\C)$, we must have that $\phi(a_1)$ and $\phi(a_2)$ do not commute.
Since $[\phi(b), \phi(c)] \neq 1$ and $\SL_2(\C)$ is commutative transitive away from its center, we have that $\phi(a_1)$ and $\phi(a_2)$ must commute.
Thus, $\phi([[a_1, a_2], [b,c]]) = 1$.
\end{proof}

For our next result, note that the Lamplighter group $\Z/2\Z \wr \Z$, mentioned in Remark \ref{LamplighterGroupRemark}, is not linear but has an element of order two with infinite conjugacy class.
Thus, dropping the linearity hypothesis from Theorem \ref{NotWeaklySL2Theorem} is not possible.
Further, we note that as a consequence of the following theorem $\SL_n(\Z)$, the infinite dihedral group, and any triangle group are not almost $\SL_2(\C)$-free.

\begin{thm} \label{NotWeaklySL2Theorem}
Let $\Gamma$ be a finitely generated linear group which has infinitely many elements of order two.
Then $\Gamma$ is not almost $\SL_2(\C)$-free.
\end{thm}

\begin{proof}
Let $X$ be the set of all order two elements in $\Gamma$.
Any image of $\Gamma$ in $\SL_2(\C)$ must take $\left< X \right>$ into $Z(\SL_2(\C))$ and so if $\left< X \right>$ contains a nontorsion element the group $\Gamma$ cannot be almost $\SL_2(\C)$-free.
Thus, for the remainder of the proof, we assume that $\left< X \right>$ does not contain any nontorsion element.

Since $\Gamma$ is a finitely generated linear group, by Selberg's lemma, $\Gamma$ contains a finite-index normal subgroup $\Delta \leq \Gamma$ that is torsion-free.
Since $\left< X \right>$ does not contain any nontorsion elements, we have $\left< X \right> \cap \Delta = \{ 1 \}$.
Thus $\left< X \right>$ embeds into the finite group $\Gamma / \Delta$.
This is impossible as $X$ contains infinitely many distinct elements.
\end{proof}

%
Before proving our next result concerning $G$-freeness of torsion-free virtually abelian groups, we need two technical lemmas.

\begin{lemma} \label{vAbelianContainsSolvable}
Let $\Gamma$ be a torsion-free and virtually abelian group that is not abelian. Then $\Gamma$ contains a solvable subgroup that is not abelian.
\end{lemma}

\begin{proof}
If $\Gamma$ is solvable, then there is nothing to prove. Hence, we assume that $\Gamma$ is not solvable.
In any group, the intersection of all possible conjugates of any finite-index subgroup is finite-index. Thus, since $\Gamma$ is virtually abelian, it follows that $\Gamma$ contains a normal subgroup, $N$, of finite-index that is abelian. Let $M$ be a maximal normal solvable subgroup of $\Gamma$ that contains $N$.  $\Gamma/M$ cannot be solvable as $\Gamma$ is not solvable. Thus, $\Gamma/M$ is not solvable. By Feit-Thompson, $\Gamma/M$ must have an element of even order that is not central. Let $t$ be such an element. Then there exists a natural number, $k$, such that $t^k$ has order two. Suppose that all order two elements commute in $\Gamma/M$, then the group generated by all order two elements is a nontrivial ($t^k$ is in it), proper (because $\Gamma/M$ is not abelian), abelian, and normal subgroup of $\Gamma/M$, which is impossible by maximality of $M$. Thus, there exists two order two elements in $\Gamma/M$ that do not commute. We lift these elements to $\Gamma$
and let $D$ be the group they generate. There is a short exact sequence,
$$
1 \mapsto D \cap M \mapsto  D  \mapsto S  \mapsto 1,
$$
where $S$ is solvable, as it is generated by two order two elements and hence is the image of an infinite dihedral group, which is solvable. It follows that $D$ is solvable and nonabelian, so we are done.
\end{proof}

\begin{lemma} \label{InversesLemma}
Let $\Gamma$ be a group with normal abelian torsion-free subgroup $\Delta \leq \Gamma$.
Let $t \in \Gamma$ and $a \in \Delta$ be elements such that $S = \{ t^k a t^{-k} : k \in \Z \}$ is a set of cardinality $p$ with $1<p<\infty$.
Then $\Gamma$ is not $\SL_n(\C)$-free for any $n > p$.
\end{lemma}

\begin{proof}
We use additive notation for $\Delta$ and write $a^t$ for $t a t^{-1}$ when $a\in \Delta$ and $t\in \Gamma$.
Set $S' = \{ a^{t^k} - a^{t^{k+1}}\mid k\in \Z\}$.
Thus $S'$ forms a single orbit under $\langle t\rangle$, and $a - a^t\neq 0$, so all elements of the orbit are of infinite order.
We have
\begin{equation} 
\sum_{b \in S'} b = \sum_{c\in S} (c- c^t) = \sum_{c\in S}c - \sum_{c\in S} c = 0.\label{TelescopingEquation}
\end{equation}

With this in hand, suppose, for the sake of a contradiction, that $\Gamma$ is $\SL_n(\C)$-free for some $n > p$.
Let $V$ be the subvariety of $\SL_n(\C)$ consisting of elements that are not regular semisimple along with elements $A \in \SL_n(\C)$ with the following property: 
there exists $x_1, \ldots, x_p$, a collection of complex numbers, each of which is an eigenvalue of $A$, with $x_1 x_2 \cdots x_p = 1$.
Since $V$ is a proper subvariety (guaranteed as $p < n$), we have that there exists some $\phi : \Gamma \to \SL_n(\C)$ such that $\phi(m) \notin V$.
However, $\prod_{A \in \phi(S')} A = 1$ by  (\ref{TelescopingEquation}).
Since all the elements in $\phi(S')$ commute with one another, they are simultaneously diagonalizable, so $A \in V$, which is impossible.
\end{proof}

\begin{thm} \label{VirtuallyAbelianNotSLnFree}
Let $\Gamma$ be a torsion-free and virtually abelian group that is not abelian.
Then $\Gamma$ is not $\SL_n(\C)$-free some $n$.
\end{thm}

\begin{proof}
By Lemma \ref{vAbelianContainsSolvable}  and Lemma \ref{ContainmentLemma} we reduce to the case of $\Gamma$ a solvable group.
Thus, we have a short exact sequence
$$
1 \to A \to \Gamma \to H \to 1
$$
where $A$ is abelian and $H$ is a finite nontrivial solvable group.
Let $A'$ be a maximal normal abelian subgroup containing $A$. And let $H'$ be the quotient $\Gamma/A'$.
Since $\Gamma$ is not abelian, we have that $H'$ is nontrivial.
Let $D$ be a nontrivial abelian normal subgroup of $H'$ (this exists because $H'$ is solvable).
If $D$ acts nontrivially on $A'$, then there exists some $t \in D$ and $a \in A'$ such that $tat^{-1} \neq a$.
Since $A'$ is of finite-index in $H$, we are in the situation of Lemma \ref{InversesLemma}.
In fact, if $A'$ is noncentral, then we are in the same situation.
We assume, then, that $A'$ is in the center of $\Gamma$.
Then $D$ acts trivially on $A'$, so by maximality of $A'$, we have that there exists $d_1, d_2 \in D$ such that $[d_1, d_2] \in A' \setminus \{1\}$.
Suppose, for the sake of a contradiction, that $\Gamma$ is $\SL_2(\C)$-free.
Then we may find some image of $\Gamma$ that detects $[d_1, d_2]$ in $\SL_2(\C)$ rel $Z(\SL_2(\C))$.
Since $\SL_2(\C)$ is commutative transitive away from its center, this is impossible.
Thus, $\Gamma$ is not $\SL_2(\C)$-free.
\end{proof}

\begin{cor} \label{KleinBottleGroupCorollary}
Let $S$ be the Klein bottle.
Then $\pi_1(S, \cdot)$ is not in $\Lf$.
\end{cor}

\begin{proof}
We have $\pi_1 = \left< x, y : x^2 y^2 \right>$.
This group is torsion-free and virtually abelian but nonabelian, so it cannot be in $\Lf$ by Theorem \ref{VirtuallyAbelianNotSLnFree}.
\end{proof}

\begin{thm}
\label{dim-criterion}
If $\Gamma\in \Lf$ is finitely generated, non-trivial, and not isomorphic to $\Z$, then
\begin{equation}
\label{dim-ineq}
\dim \Hom(\Gamma,G) \ge \dim G+\rk G
\end{equation}
for all simply connected semisimple $G$.
\end{thm}

\begin{proof}
If $\Gamma$ is abelian, it must be free abelian of rank $r\ge 2$.  If $\gamma_1,\gamma_2$ are two generators, then $e_{G,\gamma_1}$ is surjective, and the fiber
over $g\in G(\C)$ maps onto $C_G(g)$ under $e_{G,\gamma_2}$.  Since the dimension of every centralizer is at least $\rk G$ \cite[\S1.6]{MR1343976},
this implies (\ref{dim-ineq}).  We can therefore, by Lemma~\ref{SolvableGroupLemma} and
Theorem \ref{VirtuallyAbelianNotSLnFree}, assume that $\Gamma$ is not virtually solvable,
and every finitely generated subgroup of $\Gamma$ which is virtually solvable is abelian.  

We say $g\in G(\C)$ is  \emph{torus-generic} if $g$ is regular semisimple and $g$ generates a Zariski-dense subgroup of $T := C_G(g)$ (which is a maximal torus by Steinberg's theorem \cite[Th.~2.11]{MR1343976}).
We claim that the torus-generic elements lie in the union of countably many proper closed subvarieties of $G$.  
Indeed, there are  countably many closed subgroups $S$ of a maximal torus $T$
(\cite[Cor.~8.3]{MR1102012}), and for each such $S$,
the set of conjugates of elements of $S$ is contained in the proper closed subvariety of $G$ which is the Zariski-closure of the conjugation map $\xi\colon G\times S\to G$.
All fibers of this morphism have dimension $\ge \dim T$ because if $(h,s)\in \xi^{-1}(g)$, then $(hT,s)\subset \xi^{-1}(g)$.  By Theorem~\ref{dim-variation}, 
$$\overline{\xi(G\times S)}\le \dim G + \dim S - \dim T < \dim G.$$

Suppose that there exists a homomorphism $\phi\colon \Gamma\to G(\C)$ and elements $\gamma_1,\gamma_2,\gamma_3\in \Gamma$ such that
\begin{enumerate}
\item $\phi$ is a regular point of an irreducible component $\Omega$ of $\Hom(\Gamma,G)$
\item The restriction of $e_{G,\gamma_1}$ to $\Omega$  induces a surjection of tangent maps $T_\phi\Omega \to T_{\phi(\gamma_1)} G$.
\item $\phi(\gamma_i)$ is torus-generic for all $\gamma_i$ for $i\in \{1,2,3\}$ 
\item  $\gamma_3$ lies in the derived group of $\langle \gamma_1,\gamma_2\rangle$.
\end{enumerate}
\noindent
We claim that these conditions imply (\ref{dim-ineq}).  Any element of $G(\C)$ which stabilizes $\phi$ must commute with $\phi(\gamma_1)$ and $\phi(\gamma_2)$.
Thus, if $\Stab_G(\phi)$ has positive dimension, there exists an element $g\in G(\C)$ of infinite order such that $g$ commutes with both of these
elements.  As $\phi(\gamma_1)$ is regular semisimple, its centralizer consists only of semisimple elements, so $g$ is semisimple.  As $\langle g\rangle$ lies in the
center of the connected reductive group $H := C_G(g)$, it follows that the center of $H$ is positive-dimensional and therefore that the derived group $D$
of $H$ has rank strictly smaller than $\rk H$ and therefore strictly smaller than $\rk G$.  However, $\gamma_3\in D(\C)$ is torus-generic, so this is impossible.
Thus, we have the diagram
$$\xymatrix{O_G(\phi)\ar@{^{(}->}[r]\ar[d]_{e_{G,\gamma_1}}&\Omega \ar[d]^{e_{G,\gamma_1}} \\
O_G(\phi(\gamma_1))\ar@{^{(}->}[r]& G,}$$
where $O_G$ denotes $G$ orbit.  As conjugacy classes are non-singular varieties, the tangent space to $O_G(\phi(\gamma_1))$ at $\phi(\gamma_1)$
has dimension $\dim G - \rk G$.  As
$e_{G,\gamma_1\,\ast}$ maps $T_\phi O_G(\phi)\cong \g$ to $T_{\phi(\gamma_1)}O_G(\phi(\gamma_1))$, a space of dimension $\dim G - \rk G$, its kernel must have dimension at least $\rk G$, and as 
$e_{G,\gamma_1\,\ast}$ is a surjective map $T_\phi\Omega \to T_{\phi(\gamma_1)} G$, it follows that 
$$\dim  T_\phi \Omega\ge \dim G + \rk G.$$
By definition of regularity, Condition (1) now implies the theorem.

To construct $\gamma_i$ as above, we use induction on $n$ to prove the following claim:  if $\Delta$ denotes the free group on two generators, $x$ and $y$,
and $\Delta_1,\ldots,\Delta_n\subset \Delta$  have union $\Delta^\bullet$, then for some $i$ there exist
elements, $\gamma_1,\gamma_2\in \Delta_i$, and some element $\gamma_3$ in the intersection of $\Delta_i$ and 
the derived group of $\langle \gamma_1,\gamma_2\rangle$.
The case $n=1$ is trivial.  If the statement is true for $n$ and $\Delta = \Delta_1\cup\cdots\cup \Delta_{n+1}$, then we observe that as 
no two of the elements $x,xy,xy^2,\cdots$ commute with one another,  at least one of the sets $\Delta_i$, without loss of generality $\Delta_{n+1}$, contains elements
$\alpha,\beta$ which fail to commute.  If  $\Delta_{n+1}$ contains some element of the derived group of $\langle \alpha,\beta\rangle$, then we are done.
If not, we replace $\Delta$ by any subgroup $\Delta'$ of $[\langle \alpha,\beta\rangle,\langle \alpha,\beta\rangle]$ generated by two non-commuting elements
and replace $\Delta_1,\ldots,\Delta_n$ by $\Delta'_i := \Delta_i \cap \Delta'$.  The claim now follows by induction.

By Theorem~\ref{ContainsFreeSubgroupTheorem}, $\Gamma$ contains a subgroup $\Delta$ isomorphic to the free group on two generators.
We define the $\Delta_i$ to be the intersections of $\Delta$ with $S_\Omega$, as defined in Proposition~\ref{LinearEmbeddingLemma}, as $\Omega$
ranges over the irreducible components of $\Hom(\Gamma,G)$.  We fix $\Omega$ such that there exist $\gamma_1,\gamma_2,\gamma_3\in S_\Omega\cap \Delta$,
with $\gamma_3$ in the derived group of $\langle \gamma_1,\gamma_2\rangle$.  Conditions 1 and 2 on $\phi$ 
are non-empty and open and Condition (3) is satisfied on the complement of a countable union of proper closed subvarieties of $\Omega$.  
By Theorem~\ref{Baire}, all three can be satisfied simultaneously, and the theorem follows.

\end{proof}

\begin{cor}
If $w\in F_2$ is a word such that for some simply connected semisimple group $G$, the word map $G^2\to G$ is flat over some neighborhood of the identity in $G$,
then the one-relator group $\Gamma$ determined by $w$ is not in $\Lf$.
\end{cor}

\begin{proof}
Applying Theorem~\ref{dim-variation} to  $e_{G,w}\colon G^2\to G$, we deduce that the dimension of $w^{-1}(1) = \Hom(\Gamma,G)$ is $\dim G < \dim G +\rk G$.
\end{proof}

We also have a version of Theorem~\ref{dim-criterion} for $\Lt$:

\begin{thm}
\label{no-vss}
Suppose $\Gamma\in \Lt$ is finitely generated and linear over $\C$.  If every virtually solvable subgroup of $\Gamma$ is cyclic, but $\Gamma$ itself is not, then
$$\dim \Hom(\Gamma,G) \ge \dim G+\rk G$$
for all simply connected semisimple $G$.
\end{thm}

\begin{proof}
By the Tits alternative, every non-cyclic subgroup of $\Gamma$ contains a free subgroup on two generators $x$ and $y$.  For any semisimple $G$,
the union of the sets $S_\Omega$ associated with the irreducible components $\Omega$ of $\Hom(\Gamma,G)$ consists of all nontorsion elements of 
$\Gamma$.  Therefore, we can find two elements of the form $xy^i, xy^j$ which belong to the same $S_\Omega$.  The proof now finishes in exactly the same
way as the proof of Theorem~\ref{dim-criterion}.
\end{proof}

With a little work, we can now deduce Theorem~\ref{TriangleGroupTheorem}.

\begin{proof}[Proof of Theorem~\ref{TriangleGroupTheorem}]
Every von Dyck group is of the  form
$$\Gamma =  \langle x_1,x_2,x_3: x_1^a,x_2^b,x_3^c,x_1 x_2 x_3\rangle,$$
where $a,b,c$ are positive integers with $1/a+1/b+1/c < 1$.  As $\Gamma$ embeds in $\PSL_2(\R)$, it is linear, so it suffices to show that 
all of its virtually solvable subgroups are cyclic and that
$\dim \Hom(\Gamma, \SL_2) \leq 3.$

Let $\Delta\subset \Gamma$ be a virtually solvable subgroup, and let $H\subset \PGL_2$
denote its Zariski-closure.  As $H$ is virtually solvable, it is 
a proper closed subgroup of $\PGL_2$ and therefore either zero-dimensional, 
contained in a Borel subgroup, or contained in the normalizer of a maximal torus.
Since every finite subgroup of a von Dyck group is cyclic, we need only consider the two remaining cases.
Since $\Gamma$ is a discrete and cocompact subgroup of $\PSL_2(\R)$, it has no non-trivial unipotent elements \cite[Theorem 4.2.1]{MR1177168}, so if $H$
is contained in a Borel, it is contained in a maximal torus.  Thus, we can assume 
a subgroup $\Delta^\circ$ of $\Delta$ of index $\le 2$ is contained in a maximal torus.

The fact that $\Gamma$ contains a hyperbolic surface group as a subgroup of finite index implies that $\Delta^\circ$ contains a finite index
subgroup which is abelian and embeds in a hyperbolic surface group.  This subgroup must be isomorphic to $\Z$.  Thus $\Delta^\circ$ is an abelian group containing $\Z$ as a subgroup of
finite index.  Since the elements of finite order in $\Gamma$ all have finite centralizers, $\Gamma^\circ$ is torsion-free and therefore isomorphic to $\Z$.
The only groups which contain $\Z$ as a subgroup of index $\le 2$ are $\Z$ itself, $\Z\times \Z/2$, the infinite dihedral group, and the Klein bottle group.
Theorem~\ref{NotWeaklySL2Theorem} rules out the second and third cases.  The Klein bottle group cannot embed 
in a discrete group of orientation-preserving isometries of the hyperbolic plane, since the quotient would be an orientable surface whose $\pi_1$ is the Klein bottle
group, which is impossible.  Thus, $\Delta$ is, indeed, cyclic.

For the dimension computation, we note that for all homomorphisms $\phi$ with $\phi(x_1) = I$ for some $i$, $\phi$ is determined 
by $\phi(x_2)$, so the subvariety of the representation variety satisfying this condition has dimension $\le 3$.  Likewise for $x_2$ and $x_3$, and likewise if some $\phi(x_i) = -I$.  Therefore, we can assume that $\phi(x_i)$ has order $\le 3$ for all $i$, and therefore the eigenvalues $\lambda_i^{\pm 1}$ are distinct from one another.
There are finitely many possibilities for $\lambda_1,\lambda_2,\lambda_3$, given the values $a,b,c$.  We fix these values and show that the resulting subvariety $X$ has dimension $\le 3$.  Indeed, $e_{x_1,\SL_2}$ maps $X$ to the (two-dimensional) conjugacy class $Y$ of the diagonal matrix 
$$D_{\lambda_1} := \begin{pmatrix} \lambda_1&0 \\ 0&\lambda_1^{-1}\end{pmatrix}.$$
As this morphism respects the conjugation action of $\SL_2$, and $Y$ consists of a single $\SL_2$-orbit, the dimension of the fibers $X_y$ does not depend on $y$.
We therefore consider the fiber $X_{D_{\lambda_1}}$.  By Theorem~\ref{dim-variation}, it suffices to prove that this dimension is $1$.  
To do this, we note that if
$$\phi(x_2) = \begin{pmatrix}z_{11}&z_{12} \\ z_{21}&z_{22}\end{pmatrix}$$
then $z_{11}$ and $z_{22}$ are uniquely determined by the conditions $\trace(\phi(x_2)) = \lambda_2+\lambda_2^{-1}$ and $\trace(\phi(x_3)) = \lambda_3+\lambda_3^{-1}$.  The product $z_{12}z_{21}$ is then uniquely determined by the determinant $1$ condition on $\phi(x_2)$.  As $\phi$ is determined by $\phi(x_1)$ and $\phi(x_2)$,
this implies that $\dim X_{D_{\lambda_1}}=1$.

\end{proof}

We remark that one could bypass Theorem~\ref{no-vss} here by proving directly that there are finitely many $\SL_2(\C)$ orbits of homomorphisms
$\Gamma\to \SL_2(\C)$.  However, the argument above gives a non-trivial example in which the hypotheses of Theorem~\ref{no-vss} can be checked.

\section{An application to double word maps}
\label{ApplicationsSection}

Motivated by the desire to show that generic random walks on a finite simple group never live inside an algebraic subgroup, Breuillard, Green, Guralnick, and Tao prove the following 

\begin{thm} (\cite{BGGT10})
Let $w_1, w_2$ be two elements in a free group of rank $2$.
Let $a, b$ be generic elements of a semisimple Lie group $G$ over an algebraically closed field. Then $w_1(a,b)$ and $w_2(a,b)$ generate a Zariski-dense subgroup of $G$.
\end{thm}

This theorem allows Breuillard, Green, Guralnick, and Tao to prove results on \emph{expanding generators} (elements that generate a Cayley graph which is an expander) in groups of Lie type in a companion paper \cite{BGGT13}.
In light of their result, they ask the following:

\begin{ques}[Problem 2 in \cite{BGGT10}]
Can one characterize the set of pairs of words $(w_1, w_2)$ in the free group $F_2$ such that the double word map $G\times G \to G \times G$ given by 
$$e_{w_1,w_2}(a,b)  =  (w_1(a,b),w_2(a,b))$$
is dominant? 
\end{ques}

\noindent
In response to their problem we supply:
\newcommand{\bfa}{\mathbf{a} }

\begin{proof}[Proof of Theorem \ref{MainAnswerThm}]
It suffices to prove the theorem for $G$ simply connected and simple.
Let $w_1$ and $w_2$ be fixed, and let $\phi := e_{w_1,w_2}$  denote the double word map associated to $(w_1,w_2)$.
It suffices to find a point $\bfa := (a_1,b_1,\ldots,a_k,b_k)\in G^{2k}(\C)$ at which the map $\phi_*$ of tangent spaces is surjective.
By Theorem~\ref{smooth-vs-dominant}, this implies that $\phi$ is dominant.

Let $\Gamma$ denote the quotient of the free group in $2k$ generators $x_i$ by the normal subgroup generated by the conjugacy class of
$w_1$, and let $\bar w_2\neq 1$ denote the image of $w_2$ in this group.
We can identify the homomorphism variety $\Hom(\Gamma,G)$ with the variety $e_{w_1}^{-1}(1)$.  This is an irreducible variety of dimension $(2k-1)\dim G$ \cite{MR0169956,MR1206154}, and we choose a regular point $\bfa \in w_1^{-1}(1)\subset G^{2k}(\C)$
at which the evaluation map $e_{\bar w_2}$ induces a surjective map of tangent spaces
$\Hom(\Gamma,G) = e_{w_1}^{-1}(1)\to G$.  Indeed, surface groups are residually free and therefore
$G$-free, so $e_{\bar w_2}$ is dominant.  By Theorem~\ref{smooth-vs-dominant}, there is an open set on which $e_{\bar w_2}$ induces a surjection of tangent
spaces.

By Theorem~\ref{CM}, $e_{w_1\,\ast}\colon T_{\bfa} G^{2k}\to T_1 G$ is surjective.
We consider the map 
$$\phi_*\colon T_{\bfa}(G^{2k})\to T_{(1,\bar w_2(\bfa))}(G\times G) = T_1(G)\oplus T_{ w_2(\bfa)}(G).$$
Composition with projection unto the first summand gives the map $e_{w_1\,\ast}$
To prove the surjectivity of $\phi_*$, it suffices to prove that at $\bfa$, 
$\phi_* \ker e_{w_1\,\ast}$ projects  onto $T_{ w_2(\bfa)}(G)$.  It is clear that $T_{\bfa} e_{w_1}^{-1}(1)$ lies in the kernel of $e_{w_1\,\ast}$,
and we have chosen $\bfa$ such that $e_{\bar w_2\,\ast}$ is surjective.  The theorem follows.

\end{proof}

\noindent
It is natural to ask whether the constraints on $w_1$ can be weakened.
In particular, if we can handle words $w_1$ is of the form
$$h(x_1, \ldots, x_k) h(x_{k+1}, \ldots, x_{2k}).$$ 

\noindent

\appendix
\section{Algebraic geometry background}

\label{AlgebraicGeometryBackgroundAppendix}

We collect here some basic terminology and known facts regarding algebraic varieties, with special reference to
representation varieties.

By a \emph{variety} $X$ over a field $K$, we mean in this paper, an affine scheme of finite type over $K$, i.e. $X = \Spec A$, where
$A$ is a finitely generated $K$-algebra, the \emph{coordinate ring of $X$}.  Recall that as a set, $X$ consists of the prime ideals of $A$.  It is 
endowed with the Zariski-topology, where the closed
sets 
$$V(I) = \{P\in \Spec A\mid I \subset P\} $$
are in one-to-one correspondence with radical ideals $I = \rad(I)$ of $A$.  
Maximal ideals $\m$ of $A$ correspond to points of $X$ which are closed in the Zariski topology.
The set of closed points of $X$ is Zariski-dense in $X$ (\cite[\S4.5]{MR1322960}); note that this is a property of spectra of finitely generated algebras over
a field and is not true for general affine schemes.
The \emph{Zariski tangent space} $T_x(X)$ at a closed point $x$ in $X$ corresponding to 
a maximal ideal $\m$ of $A$ is the linear dual of $\m/\m^2$ regarded as a vector space over $A/\m$.

A topological space is \emph{irreducible} if it cannot be realized as a finite union of proper closed subsets.
Thus, the closed set $V(I)$ is irreducible if and only $I$ is a prime ideal.
By the Hilbert basis theorem, $A$ is a Noetherian ring; the ascending chain condition on ideals implies 
the descending chain condition on closed subsets, so the process of decomposing closed subsets of $\Spec A$ into finite unions of
proper closed subsets must terminate, and $\Spec A$ can be written as a finite union of irreducible components.
Each component can be regarded as a variety in its own right, namely $\Spec A/P$, where $P$ is the prime ideal associated to the component.
Every prime ideal of $A$ contains the nilradical $\rad A$, and $X$ is irreducible if and only if $\rad A$ is prime, in which case
it is the unique minimal prime ideal, which will be denoted $\eta\in X$.  If $I\subset \rad A$, every prime ideal of $A$ contains $I$, so that from a topological point of view
there is no difference between $\Spec A$ and $\Spec A/I$.

If $L$ is an extension field of $K$, we denote by $X(L)$ the set of $L$-points of $X$, i.e., the set of $K$-homomorphisms $A\to L$.
In particular, a linear algebraic group (and all algebraic groups in this paper are assumed to be linear) is a $K$-variety, whether or not
it is connected.  If $K$ is algebraically closed, then by Hilbert's Nullstellensatz, $X(K)$ is identified with the closed points of $X$, so $X(K)$ is Zariski-dense in $X$.
We will usually consider the case $K=\C$, and we sometimes fail to distinguish between $X$ and $X(\C)$ when this is unlikely to cause confusion.

A \emph{morphism of varieties} $X = \Spec A\to Y = \Spec B$ is  a homomorphism $f\colon B\to A$ of $K$-algebras.
If $y\in Y$ corresponds to the prime ideal $P$ of $B$, then the \emph{fiber} of $f$ at $y$, denoted by $X_y$, is the variety over $\Frac B/P$ associated to $A\otimes_K\Frac B/P$,
where $\Frac$ denotes the field of fractions of an integral domain.

\begin{thm}
If $f\colon X\to Y$ is a morphism and $X$ is irreducible, then the Zariski closure of $f(X)$ is irreducible.
\end{thm}

\begin{proof}
As $X$ is irreducible, the point $\eta$, corresponding to the nilradical of the coordinate ring of $A$, is Zariski dense in $X$.
The closure of any point of a topological space is irreducible, so $\overline{f(\eta)}$ is an irreducible subset of $Y$.  As 
$$\{f(\eta)\} \subset f(X)\subset \overline{f(\eta)},$$
taking closures, we obtain $\overline{f(X)} = \overline{f(\eta)}$.
\end{proof}

Let  $\Gamma$ be a group, $K$ a field, and $G$ an algebraic group over $K$.
If $X$ is a $K$-variety with coordinate ring $A$,
$\Phi\colon \Gamma\to G(A)$ is a homomorphism, and $B$ is a $K$-algebra, every element of $X(B)$ determines a homomorphism
$A\to B$, therefore a homomorphism $G(A)\to G(B)$, and by composition with $\Phi$, a homomorphism $\Gamma\to G(B)$.
If $\Gamma$ is finitely generated, then
there exists a finitely generated $K$-algebra $A$ and a homomorphism $\Phi\colon \Gamma\to G(A)$ which is universal in the sense that
for all $K$-algebras $B$, there is
a one-to-one correspondence between homomorphisms $\Gamma\to G(B)$ 
and elements of $X(B)$.
(for a proof, see \cite[Prop.~1.2]{MR818915}, where a valid general argument
is given under the unnecessary hypotheses that
$K$ is algebraically closed and of characteristic zero and $G = \GL_n$.)
When $\Gamma$ is a free group on $d$ generators, $X = G^d$.  More generally,
if $\Gamma$ is a quotient of $F_d$, we can regard $X$ the closed subvariety of $G^d$
given by the conditions that each relation word $\gamma\in \ker F_d\to \Gamma$
maps to the identity in $G$.  In particular, for a group $\Gamma$ given by a single relation $\gamma\in F_d$,
we can identify $\Hom(\Gamma,G)$ with the fiber of the evaluation morphism $e_{F_d,\gamma}\colon G^d\to G$
at the identity $1\in G(K)$.

We observe (see \cite{MR818915}) that this construction works over a general commutative ground ring $R$, so if $G$ is an affine group
scheme over $R$, then $\Hom(\Gamma,G)$ is represented by $\Spec A$ for some finitely generated $R$-algebra $A$.  This construction respects change of base ring
so for example if $R =\Z$, $G = \SL_{n,\Z}$, and $\Hom(\Gamma,G)$ is represented by a $\Z$-algebra A, then for all primes $p$, $\Hom(\Gamma,\F_p)$ is represented by $A\otimes\F_p$.

If $Y$ is an irreducible variety, we say a morphism $f\colon X\to Y$ is \emph{dominant} if $f(X)$ is Zariski-dense in $Y$.  This is equivalent to the
statement that $f(\Omega)$ is dense in $Y$ for some irreducible component $\Omega$ of $X$.  The following theorem is an immediate consequence of Chevalley's theorem \cite[Cor.~14.7]{MR1322960}:
\begin{thm}
\label{constructible}
The image of a dominant morphism $X\to Y$ contains a non-empty open subset of $Y$.
\end{thm}

By the \emph{dimension} of $X=\Spec A$, we mean the Krull dimension of $A$, which is the maximum
possible length $n$ of a chain of prime ideals $P_0\subsetneq P_1\subsetneq \cdots\subsetneq P_n$ of $A$.
Note that the dimension of a variety is the maximum of the dimensions of its components.
For an irreducible variety $\Spec A$, the dimension of $X$ equals the transcendence degree of $\Frac A/\rad A$ \cite[\S8.2.1, Th.~A]{MR1322960}.
Given a morphism $X\to Y$, the function $y\mapsto \dim X_y$ is not in general constant on irreducible components of $Y$.
However, we have the following facts:

\begin{thm}
\label{dim-variation}
Let $f\colon X\to Y$ be a morphism of irreducible varieties.  Then:
\begin{enumerate}
\item The function which assigns to each $x\in X$ the dimension of the irreducible component of $X_{f(x)}$ to which $x$ belongs is
upper semicontinuous.
\item If $f$ is dominant, there exists a non-empty open set of $Y$ on which every irreducible component of $X_y$ has dimension $\dim X - \dim Y$.
\item If $f$ is dominant and flat in a neighborhood of $y$ (i.e., $B[1/b]$ is flat as an $A[1/\phi(b)]$-module for some $b$
which does not lie in the prime ideal associated to $y$), then $\dim X_y = \dim X - \dim Y$.
\item Every component of $X_y$ has dimension at least $\dim X - \dim Y$.
\end{enumerate}
\end{thm}

\begin{proof}
The first part is a special case of \cite[Th.~14.8a]{MR1322960}.  
For the remaining statements,  let us first assume that $f$ is dominant,
i.e., we have a homomorphism $\phi\colon B\to A$ for which $f^{-1}(\rad A) = \rad B$.
Replacing $A$ and $B$ by $A/\rad A$ and $B/\rad B$ respectively and $\phi$ by the induced homomorphism
$A/\rad A\to B/\rad B$, $A$ and $B$ become integral domains and $\phi$ becomes an injection, but at the level of points nothing changes.
So, we can apply \cite[Cor.~13.5]{MR1322960} and deduce that the dimension of the generic fiber (which is irreducible since $X$ is) 
satisfies $\dim X_\eta = \dim X - \dim Y$.  Part 1 now implies parts 2 and 4  in the dominant case.  Part 3 follows from part 2 and \cite[Th.~10.10]{MR1322960}.
For part 4, if $f$ is not dominant, we can replace 
$B$ by $B/\ker \phi$, which means, topologically, that $Y$ is replaced by $\overline{f(X)}$.  The general statement follows.

\end{proof}

The following result is an algebraic analogue of the Baire category theorem.

\begin{thm}
\label{Baire}
Let $K$ be an algebraically closed field and $J$ a set whose cardinality is strictly less than that of $K$.
If $X$ is an irreducible variety over $\C$ and $\{X_j\mid j\in J\}$ a  collection of proper closed subvarieties, then
$$X(K)\setminus \bigcup_{j\in J}X_j(K)$$
is non-empty.
\end{thm}

\begin{proof}
Without loss of generality we can assume that $X$ is irreducible.  Replacing $A$ by $A/\rad A$ does not change the underlying set,
so without loss of generality we can assume $A$ is an integral domain. 
We use induction on $\dim X$.  If $\dim X=1$, every proper subvariety has a finite number of $0$-dimensional
components, so it suffices to prove that $|X(K)| \ge |K|$.   Let $a$ denote any element of its fraction field which 
is not in $K$.  Then $t\mapsto a$ defines a morphism $\Spec A\to \Spec K[t]$.  The Zariski closure of the image is irreducible, hence either a point
or all of $\Spec K[t]$.  As $a$ is not constant, Theorem~\ref{constructible} implies
that the image contains all but finitely many elements of $K$.  Since $K$ is infinite, the cardinality of the image
equals that of $K$, which finishes the base case.

For the induction step, we again use a  morphism $X\to \Spec K[t]$ given by a non-constant element $a$.  Again the image contains all but
finitely many points of $\Spec K[t]$ and by Theorem~\ref{dim-variation},
all components of all non-empty fibers of this morphism have dimension $\dim X-1$.
Therefore, no irreducible proper subvariety of $X$ can contain a component of more than one fiber, and no proper subvariety of $X$ can contain
components of more than a finite number of fibers.  It follows that some non-empty fiber has an irreducible component $X'$ not contained in any of the $X_j$.
Replacing $X$ by $X'$ and each $X_j$ by $X'\cap X_j$, the theorem follows by induction.
\end{proof}

In particular, we apply this theorem in the case that $K=\C$ and $J$ is countable.

\begin{thm}
\label{geom-irred}
The following three conditions on a finitely generated $K$-algebra $A$ are equivalent:
\begin{enumerate}
\item $\Spec A\otimes_K L$ is irreducible for all finite $K$-extensions $L$.
\item $\Spec A\otimes_K \bar K$ is irreducible for some algebraic closure $\bar K$ of $K$.
\item $\Spec A\otimes_K L$ is irreducible for all $K$-extensions $L$.
\end{enumerate}
\end{thm}

\begin{proof}
By \cite[Prop.~4.5.9]{MR0199181}, Condition (1) implies Condition (3).  It is trivial that Condition (3) implies Condition (2).  To see that Condition (2) implies Condition (1),
we note that $A_L:= A\otimes_K L$ fails to be irreducible for some finite extension $L$, if and only if there exist $f,g\in A_L$ which are not nilpotent and such that $fg = 0$.
Realizing $L$ as a $K$-subextension of $\bar K$, we conclude that $A_{\bar K}$ also contains non-nilpotent elements which multiply to zero.
\end{proof}

We say a variety $X = \Spec A$ over $K$ is \emph{geometrically irreducible} if $A$ satisfies any of these equivalent conditions.
From Theorem~\ref{geom-irred} and the finiteness of the set of irreducible components $\Spec A_{\bar K}$, which is a variety over $\bar K$,
it follows that $K$ has some finite extension $L$ such that $A_L$ is a finite union of geometrically irreducible components.  By the proof of \cite[Cor.~4.5.11]{MR0199181},
the field of fractions of $A_{\bar K}/P$, where $P$ is any minimal prime ideal, containsis a finite extension of $\Frac A/\rad A$, so the transcendence degrees over $K$
are equal.  It follows that  the dimension of every irreducible component of $\Spec A_L$ equals $\dim A$.

\begin{thm}
\label{generic-irred}
Let $A$ be a finitely generated $\Z$-algebra.  If there exist infinitely many primes $p$ such that $A\otimes \F_p$ is a geometrically irreducible 
$\F_p$-algebra of dimension $n$, then $A\otimes \C$ is irreducible of dimension $n$.
\end{thm}
Replacing $X$ by $X^{\red} := A/\rad A$ does not change $X(L)$ for any field extension $L$ of $K$, so if $K$ is a finite field $\F_q$, we have
$|X(\F_{q^n}) = X^{\red}(\F_{q^n})$ for all $n\in \N$.   Now, $X^{\red}$ is a variety in the sense of Lang-Weil \cite{LW54},
so by the result of that paper,
$$|X(\F_{q^n})| = q^{n\dim X}(1+o(1)).$$
If $X$ is irreducible but not geometrically irreducible, then for some $m\in \N$, 
$$X_m := \Spec A\otimes_{\F_q}\F_{q^m}$$ 
decomposes as a union of 
$c\ge 2$ irreducible varieties of dimension $n$, and
$$|X(\F_{q^{mn}})| = |X_m(\F_{q^{mn}})| = c q^{mn\dim X}(1+o(1)).$$
For any variety, we can take $m$ sufficiently divisible that $X_m$ is a union of geometrically irreducible components, and
$$|X(\F_{q^{mn}})| = c_{\dim X} q^{mn\dim X}(1+o(1)),$$
where $c_k$ is the number of irreducible components of dimension $k$.  From this we deduce:

\begin{thm}
\label{converse-LW}
If $|X(\F_{q^n})| = (1+o(1)) q^{kn}$, then $X$ has a single component of dimension $k$, it is geometrically irreducible, and all other components have lower
dimension.
\end{thm}

A morphism $X = \Spec A\to Y = \Spec B$ corresponding to a $K$-algebra homomorphism $\phi\colon B\to A$
is \emph{flat} if $A$ is flat when regarded as a $B$-module via $\phi$.  We say $X\to Y$ is \emph{flat in a neighborhood of a point} $x\in X$
corresponding to a prime ideal $P$ if there exists $b\in B$ such that $\phi(b)\not\in P$ and $A[1/\phi(b)]$ is flat as a $B[1/b]$ module.
In particular, if $X$ is irreducible, we say $X\to Y$ is \emph{generically flat} if it is flat in a neighborhood of the generic point $\eta$ of $X$.
Note that \emph{generic} here means generically in $X$.

\begin{thm}
\label{generic-flatness}
If $A$ and $B$ are integral domains, then $\Spec A\to \Spec B$ is dominant if and only if it is generically flat.
\end{thm}

\begin{proof}
As $B$ is an integral domain if $\Spec A\to \Spec B$ is dominant, then $\phi\colon B\to A$ is injective.
Grothendieck's generic freeness lemma \cite[Th.~14.4]{MR1322960} says that there exists a non-zero $b$ such that
$A[1/\phi(b)]$ is free, and therefore flat, as a $B[1/b]$-module.  Conversely, if there exists $b\in B$
such that $A[1/\phi(b)]$ is flat as B[1/b]-module, then every non-zero element of $B[1/b]$ maps to a non-zero element of $A[1/\phi(b)]$.
Since $B$ is an integral domain, $B\to A[1/\phi(b)]$ is injective \cite[Cor.~6.3]{MR1322960}.  As this homomorphism factors through $B\to A$,
the latter morphism is injective, and the generic point of $\Spec A$ maps to the generic point of $\Spec B$.
\end{proof}

A local ring $B$ with maximal ideal $M$ is \emph{regular} if $\dim_{B/M}M/M^2 = \dim B$.  In the special case that
$B = A_{\m}$ where $\m$ is the maximal ideal corresponding to a closed point $x$, this is equivalent to the condition $\dim T_x X = \dim X$,
and we say that $X$ is \emph{regular at $x$}.  If all local rings of $A$ are regular, we say that $X$ is \emph{non-singular}.

\begin{thm}
\label{smoothness}
Suppose $K$ is a perfect field and a $K$-variety $X$ is regular at a closed point $x\in X(K)$.
Then there exists an element $f\in A$ which does not lie in the maximal ideal associated to $x$, a morphism 
$$Y = \Spec[y_1,\ldots,y_m]\to Z = Spec[z_1,\ldots,z_{m-n}],$$
and an isomorphism $i\colon \Spec A[1/f] \to Y_z$, where $z$ corresponds to the ideal 
$$(z_1,\ldots,z_{m-n}),$$
such that the induced map
of Zariski tangent spaces $T_{i(x)}Y \to T_z Z$ is surjective.
Moreover, there exists a morphism from $X$ to $\Spec K[t_1,\ldots,t_n]$ which induces an isomorphism of tangent spaces at $x$.
\end{thm}

\begin{proof}

Conversely, since $K$ is perfect, if $X$ is regular at $x$, then $Y_z\to \Spec K$ is smooth at $i(x)$ in the sense of Grothendieck \cite[Cor.~17.15.3]{MR0238860}.
By \cite[I~Prop.~3.24]{Milne}, for some $m\ge n$ and some $P_1,\ldots,P_{m-n}\in K[y_1,\ldots,y_m]$, 
there exists $f\not\in \m$ such that 
$$A[1/f]\cong K[y_1,\ldots,y_m]/(P_1,\ldots,P_{m-n})$$ 
and 
$$\det\begin{pmatrix}\frac{\partial P_i}{\partial y_j} \end{pmatrix}_{1\le i,j\le m-n}$$
is a unit in $A[1/f]$.  The $K$-homomorphism $K[z_1,\ldots,z_{m-n}]\to K[y_1,\ldots y_m]$ sending $z_i\mapsto P_i(y_1,\ldots,y_m)$
determines a morphism $Y\to Z$ whose fiber over $(z_1,\ldots,z_{m-n})$ is isomorphic to $\Spec A[1/f]$ and which induces a 
surjection of tangent spaces as claimed.  By \cite[Cor.~17.15.9]{MR0238860}, there is a morphism $g\colon X\to \Spec K[t_1,\ldots,t_n]$ which is \'etale at $x$
and therefore smooth.  Since the dimensions are equal, $g$ induces an isomorphism of tangent spaces at $x$ \cite[Th.~17.11.1]{MR0238860}.

\end{proof}

\begin{thm}
\label{density}
If $K$ is a local field, $X/K$ a variety, and $x$ a regular point of $X$ corresponding to a maximal ideal with $A/\m\cong K$, then $X(K)$
is Zariski-dense in $X$.
\end{thm}

\begin{proof}
By the implicit function theorem  \cite[Th.~5.7.1]{MR0219078}, there exists a neighborhood of $x$ in $X(K)$ which is a $K$-analytic manifold, and by
a second application of the same result, $g$ induces an isomorphism from a neighborhood of $x$ in $X(K)$ to a neighborhood of $g(i(x))=(0,\ldots,0)$ in $K^n$.
If the Zariski closure of $X(K)$ in $X$ were of dimension $<n$ the Zariski-closure $Z$ of its image in $\Spec K[t_1,\ldots,t_n]$
would be a proper closed subvariety of $\Spec K[x_1,\ldots,x_n]$.  By induction on $n$, $Z(K)\subset K^n$ cannot contain a subset of the form
$Z_1\times \cdots\times Z_n$ where all the $Z_i$ are all  infinite.  In particular, it cannot contain a non-empty open subset of $K^n$.

\end{proof}

\begin{thm}
\label{smooth-vs-dominant}
Let $f\colon X\to Y$ be a morphism of non-singular irreducible varieties over $\C$.  If $T_x X\to T_{f(x)} Y$ is surjective,
then $X_{f(x)}$ is regular at $x$ and $X\to Y$ is dominant.  
Conversely, if $X\to Y$ is dominant, there exists a non-empty open subset $U\subset X(\C)$ so that for all $x\in U$,
$T_x X\to T_{f(x)} Y$ is surjective.
\end{thm}

\begin{proof}
A non-singular irreducible variety has a coordinate ring which is an integral domain.  Indeed, if $a$ is a non-zero element of $\rad A$, where $\Spec A$ is irreducible,
then the annihilator of $a$ is contained in some maximal ideal $\m$, so $a$ maps to a non-zero nilpotent element of the regular local ring $A_{\m}$,
which is impossible \cite[Cor.~10.14]{MR1322960}.  Thus $\rad A = 0$, and with $\Spec A$ irreducible, this means $A$ is an integral domain.
In our setting, both the coordinate ring $A$ of $X$ and the coordinate ring $B$ of $Y$ are integral domains.

As $x$ and $y:=f(x)$ are regular points, $X\to \Spec \C$ and $Y\to \Spec \C$ are smooth at $x$ and $y$ respectively.  By \cite[Th.~17.11.1]{MR0238860}, surjectivity
on the level of tangent spaces now implies that $X\to Y$ is smooth at $x$, \cite[Th.~17.5.1]{MR0238860} implies that $X_y\to \Spec \C$ is smooth at $x$,
and \cite[Prop.~17.15.1]{MR0238860} implies that $X_y$ is regular at $x$.
As \cite[Th.~17.11.1]{MR0238860} implies that $X\to Y$ is flat at $x$,
\cite[Th.~11.1.1]{MR0217086} shows that $X\to Y$ is flat in a neighborhood of $x$ and hence generically.  By Theorem~\ref{generic-flatness}, the morphism is dominant.

For the converse, we use the Jacobian criterion for smoothness in the form \cite[Cor.~16.23]{MR1322960} to prove that there exists $b\in B$ such that
$A[1/b]$ is smooth over $B[1/b]$.

\end{proof}

\begin{thm}
\label{CM}
If $f\colon X\to Y$ is a morphism of non-singular irreducible varieties over $\C$, $y\in Y(\C)$, every component of $X_y$ has dimension $\dim X - \dim Y$, and $x\in X_y(\C)$ is a non-singular point, then $T_x X\to T_y Y$ is surjective.
\end{thm}

\begin{proof}
By \cite[$\mathbf{0}_{\mathrm{IV}}$~Cor.~17.1.3]{MR0173675}, $X$ and $Y$ are Cohen-Macauley, and so by \cite[Prop.~15.4.2]{MR0217086},
$f$ is flat at every point of the fiber $X_y$.  By \cite[Th.~17.5.1]{MR0238860}, it follows that $f$ is smooth at $x$, and by \cite[Th.~17.11.1]{MR0238860},
we conclude that $T_x X\to T_y Y$ is surjective.

\end{proof}

\section{Character bounds}

\label{CharacterBoundsAppendix}

Here, we present some character bounds that are used in the proof of Theorem \ref{QuadTheorem}.

\begin{lemma}
\label{Grass}
Let $n$ be an arbitrary natural number.
There exists a real number $C_n$ such that if $s$ is a natural number less than $n$ and $T : \F_q^n \to \F_q^n$ is a semisimple linear transformation whose eigenvalues all have algebraic multiplicity at most $m$, then the number of $s$-dimensional $\F_q$-subspaces of $\F_q^n$ which are fixed by $T$ is at most $C_n q^{ms}$.
\end{lemma}

\begin{proof}
Let $V = \F_q^n$ decompose as a direct sum $V_1^{a_1}\oplus\cdots\oplus V_r^{a_r}$, where
the $V_i$ are pairwise non-isomorphic irreducible $\F_q[T]$-modules of $\F_q$-dimension $b_1,\ldots,b_r$
respectively, and $a_1,\ldots,a_r\le m$.  We can identify $V_i$ with $\F_{q^{b_i}}$ in such a way that
$\F_q[T]$-submodules of $V_i^{a_i}$ correspond to $\F_{q^{b_i}}$-subspaces of $\F_{q^{b_i}}^{a_i}$.
Every subspace $W\subset V$ fixed by $T$ is a direct sum $W_1\oplus\cdots\oplus W_r$, where
each $W_i$ is a $\F_q[T]$-submodule of $V_i^{a_i}$.  For each $r$-tuple of non-negative integers $w_i$ such that
$\sum_i b_i w_i = s$,
we can classify the subspaces $W$ such that $\dim_{\F_{q^{b_i}}}W_i = w_i$ by a product of $r$ Grassmannians
$G(a_i, w_i)(\F_{q^{b_i}})$.  As
$$|G(a_i,w_i)(\F_{q^{b_i}})| = \frac{\prod_{j=1}^{a_i} (q^{b_i j}-1)}{\prod_{k=1}^{w_i} (q^{b_i k}-1) \prod_{l=1}^{a_i-w_i} (q^{b_i l}-1)},$$
we have
\begin{align*}
\log_q \prod_{i=1}^r |G(a_i,w_i)(\F_{q^{b_i}})| &= \sum_{i=1}^r b_i w_i(a_i-w_i) + o(1) \\
                                                                                  & \le \sum_{i=1}^r b_i w_i m + o(1) = ms + o(1).
\end{align*}
\end{proof}

We understand \cite{Aner} that Aner Shalev and Roman Bezrukavnikov have unpublished estimates of character values of groups of Lie types at semisimple elements which are both stronger and more general than the following result.  Our proof, however, is elementary, using only classical results on the character theory of $\GL_n(\F_q)$ due to J.\ A.\ Green and Robert Steinberg.

\begin{thm} \label{CharacterBoundTheorem}
Given an  integer $n>0$ and positive real numbers $\alpha$ and $\beta$ such that 
$$\alpha < \frac{\beta^2}{1+2\beta},$$
for all sufficiently large finite fields $\F_q$, all irreducible characters $\chi$ of $G_n := \GL_n(\F_q)$, and all semisimple elements $x\in G_n$ whose maximal eigenvalue multiplicity is
$\le \alpha n$, we have
$$|\chi(x)| \le  \chi(1)^\beta.$$
\end{thm}

\begin{proof}
Throughout the proof, we can and do assume without loss of generality that $q$ is sufficiently large in terms of $n$; the expression $o(1)$ is short for $o_q(1)$.
We also assume $\beta < 1$, since the theorem is trivial otherwise.  

We follow the notation and terminology of Green \cite{MR0072878}.  
For $s$  a positive integer, an \emph{$s$-simplex} $g$
is a $q$-Frobenius-orbit of length $s$ of complex characters of the multiplicative group
$\F_{q^s}^\times$.
We write $s = d(g)$ and call it the \emph{degree} of $s$.
By \cite[Th.~13]{MR0072878}, the irreducible characters of $G_n$ are indexed by partition-valued functions $\nu$ on the set $S$ of simplices such that
$$\sum_{g\in S} |\nu(g)| d(g) = n,$$
where $|p|$ denotes the sum of the parts of partition $p$.  Moreover,
if $g_1,\ldots,g_k$ are the simplices on which $\nu$ is supported, the character associated to $\nu$
is obtained by parabolic induction from the characters of 
$$G_{|\nu(g_1)|d(g_1)},\ldots,G_{|\nu(g_k)|d(g_1k)}$$
associated with the partition-valued functions $\nu_i$ supported at $g_i$ and such that $\nu_i(g_i) = \nu(g_i)$.
In particular, the degree of $\chi$ is at least $|G_n/P|$ where $P$ is the parabolic subgroup associated to the sequence
$$|\nu(g_1)|d(g_1),\ldots, |\nu(g_k)|d(g_k).$$
We can assume the $g_i$ to be chosen in such an order that this sequence is non-increasing.
If $\nu$ is supported on a single simplex, we say that the character is \emph{primary}.

Let $\lambda_1+\cdots+\lambda_r = n$ express $n$ as a sum of positive integers.  If $P_\lambda$ denotes the stabilizer 
in $G_n$ of a flag of $\F_q$-spaces
\begin{equation}
\label{flag}
(0)= V_0\subset V_1\subset V_2\subset \cdots\subset V_r = \F_q^n,
\end{equation}
where $\dim_{\F_q} V_i/V_{i-1} = \lambda_i$, then
$$|G_n/P_\lambda| = \frac{\prod_{i=1}^n (q^i-1)}{\prod_{i=1}^r\prod_{j=1}^{\lambda_i} (q^j-1)},$$
so 
$$\log_q |G_n/P_\lambda| = \sum_{1\le i < j\le r}\lambda_i\lambda_j + o(1) = n^2/2 - \sum_{i=1}^r \lambda_i^2 /2 + o(1).$$
In particular, if $\chi$ is a character of $G_n$ induced from some character of $P_\lambda$, then
\begin{equation}
\label{nine}
\log_q \chi(1) \ge n^2  - \sum_{i=1}^r \lambda_i^2 /2 + o(1).
\end{equation}

If $\lambda_1\ge \lambda_2\ge \cdots\ge \lambda_r$, then
\begin{equation}
\label{ten}
n^2/2 - \sum_{i=1}^r \lambda_i^2 /2 \ge \begin{cases} \lambda_1(n-\lambda_1) &\mbox{if } \lambda_1\ge n/2, \\ n^2/4&\mbox{if } \lambda_1\le n/2.\end{cases}
\end{equation}
Let $a_1,a_2,\ldots,a_k$ denote the eigenvalue multiplicities of $x$.  As $\max_i a_i \le \alpha n$ and $\sum_i a_i = n$,
we have $\sum_i a_i^2 \le \alpha n^2$.  Since the centralizer $C(x)$ of $x$ in $G_n$ is the group of $\F_q$-points of a connected reductive group of 
dimension $\sum_i a_i^2$,
\begin{equation}
\label{eleven}
\log_q |C(x)| =  \sum_i a_i^2 + o(1) \le\alpha n^2 + o(1).
\end{equation}
By Schur's lemma, $|\chi(x)| \le |C(x)|^{1/2}$, so if $\lambda_1 \le n/2$, then by (\ref{nine}), (\ref{ten}), and (\ref{eleven}), $|\chi(x)| > \chi(1)^\beta$ implies
\begin{align*}
\beta n^2/4&\le \beta\log_q \chi(1)+o(1) < \log_q |\chi(x)| + o(1)\le \frac{\log_q  |C(x)|}2 + o(1)\\
         & \le (\alpha/2)n^2 +o(1) < \frac{\beta^2}{2+4\beta} n^2+ o(1),
\end{align*}
impossible since $\beta > 0$.
Thus, we assume $\lambda_1 > n/2$.

Let $\gamma := \alpha/\beta$, so $\gamma < 1/2$.  We claim that if $q$ is sufficiently large and $|\chi(x)| > \chi(1)^\beta$, then $\lambda_1 > (1-\gamma)n$.  Indeed, if $n/2\le \lambda_1 \le (1-\gamma)n$, then 
\begin{align*}
\frac{\alpha n^2} 2 &\ge \frac{\log_q|C(x)|}2 + o(1) \ge \log_q |\chi(x)| + o(1)> \beta\log_q \chi(1) + o(1) \\
&\ge \beta\lambda_1(n-\lambda_1)+ o(1) \ge \beta n^2 \gamma(1-\gamma) + o(1) >  n^2 \alpha/2+ o(1),
\end{align*}
which is impossible for large $q$.  This justifies our claim.

We can therefore regard $\chi$ as arising from parabolic induction from an irreducible representation $\phi'$ of $G_{n-\lambda_1}$ and a primary irreducible representation $\phi''$ of $G_{\lambda_1}$, where $n-\lambda_1 < \gamma n$.  Thus,
\begin{align*}
\log_q \chi(1) &=  \log_q \phi'(1)+\log_q \phi''(1) + \lambda_1(n-\lambda_1)+o(1) \\
                         &\ge \log_q \phi''(1) + \lambda_1(n-\lambda_1)+o(1).
\end{align*}
On the other hand, $\chi(x)$ can be written as a sum of terms of the form $\phi'(x')\phi''(x'')$, where $x'\in G_{n-\lambda_1}$,
$x'' \in G_{\lambda_1}$, and $x'\oplus x''$ is conjugate to $x$.  The terms in the sum are indexed by elements of the Grassmannian of
$(n-\lambda_1)$-planes $W\subset \F_q^n$ such that $x$ preserves $W$, the action of $x$ on $W$ is conjugate to $x'$, and the
action of $x$ on $\F_q^n/W$ is conjugate to $x''$.   

The dimension estimate 
$$\log_q \phi'(1) \le \binom{n-\lambda_1}{2} + o(1)$$
can be deduced from \cite{MR0072878}, but can also be found in various forms in the literature.  (See, e.g., \cite[Th.~2.1]{MR1023808}, \cite[Th.~5.1]{MR3020739}.)
By Lemma~\ref{Grass}, we have
\begin{align*}
\log_q |\chi(x)| &\le \log_q \phi'(1) + \max_{x''}\log_q |\phi''(x'')| + \alpha (n-\lambda_1)n+o(1) \\
                          &< \frac{(n-\lambda_1)^2}2+ \max_{x''}\log_q |\phi''(x'')| + \alpha (n-\lambda_1)n+o(1) \\
                          &\le \gamma \lambda_1(n-\lambda_1)+ \max_{x''}\log_q |\phi''(x'')| + \alpha (n-\lambda_1)n+o(1) \\
                          &< \beta (n-\lambda_1)n+ \max_{x''}\log_q |\phi''(x'')|+o(1) \\
                          &<\beta \log_q |\chi(1)| + o(1),
\end{align*}
provided that
$$\log_q |\phi''(x'')|\le \beta \log_q \phi''(1) + o(1)$$
for all semisimple $x''\in G_{\lambda_1}$ with eigenvalue multiplicity less than or equal to
$$\alpha n \le \frac{\alpha\lambda_1}{1-\gamma} = \frac{\beta^2\lambda_1}{1+\beta}.$$
Replacing $n$ by $\lambda_1$ and $\phi''$ by $\chi$,
we have a statement very similar to what we originally set out to prove.
The advantage over the original statement is that we can now assume that $\chi$ is primary; the disadvantage is that the upper
bound on maximal eigenvalue multiplicity as a fraction of $n$ is $\frac{\beta^2}{1+\beta}$ instead of $\alpha = \frac{\beta^2}{1+2\beta}$.

From now on we assume that $\chi$ is associated to a partition-valued function of simplices $\nu$ supported on a single $g$.  Denoting by $s$ the degree of $g$ and setting $v = |\nu(g)|$, we have $n=sv$.  By \cite[Lemma~7.4]{MR0072878}, for each partition $\lambda$ of $v$, there exists a rational function $\{\lambda:t\}$ such that
$$\chi(1) = \Bigl(\prod_{i=1}^n (q^i-1) \Bigr)\{\nu(g):q^s\}^{-1}.$$
As this takes integer values for all prime powers $q$, it follows that the pole at $t=\infty$ of $\{\lambda:t\}^{-1}$
has order at most $\binom {v+1}2$.  Thus,
$$\log_q \chi(1) \ge \binom {n+1}2 - s\binom {v+1}2 + o(1) \ge \frac{n^2}4 + o(1)$$
if $s\ge 2$.  If $|\chi(x)| > \chi(1)^\beta$, then
\begin{align*}
\frac{\beta^2}{1+\beta}n^2 &\ge \log_q |C(x)| + o(1)> 2\log_q |\chi(x)| + o(1) \\
                                                &\ge 2\beta \log_q \chi(1) + o(1) > \frac{\beta n^2}{2} + o(1),
\end{align*}
which is impossible when $q$ is sufficiently large since $\beta < 1$.
Thus, we can assume that $s=1$, which means that after tensoring with a degree $1$ character of $G_n$ (which does not affect $|\chi(x)|$, of course), we can assume that $\nu$ is supported on the trivial simplex, i.e., $\chi$ is a \emph{unipotent} character.

Let $\chi = \chi_\lambda$ denote the unipotent character associated with any partition $\lambda_1+\cdots+\lambda_r=n$.  Let $\phi_\lambda$ be the permutation character associated with the same partition, i.e., the character associated to the action
of $G_n$ on the set of $\F_q$-flags (\ref{flag}).
It is a classical theorem of Steinberg \cite[\S2~Cor.~1]{MR0215903}
that 
$$\phi_\lambda = \sum_\mu K_{\lambda,\mu}\phi_\mu,$$
where $K_{\lambda,\mu}$ is the Kostka number associated to $\lambda$ and $\mu$, and the sum is taken over all
partitions $\mu$ of $n$.  In particular, $K_{\lambda,\mu}=0$ unless $\mu\preceq \lambda$ in the partial order of \emph{majorization}. 
(This means that $\mu_1+\mu_2+\cdots\mu_s \le \lambda_1+\lambda_2+\cdots+\lambda_s$ for all $s\le r$, with equality when $s=r$.)
It is also known that $K_{\lambda,\lambda} = 1$, and from \cite[Lemma~7.4]{MR0072878}, it follows that
$\dim \chi_\mu \le \dim \chi_\lambda$ if $q$ is sufficiently large and $\mu\preceq \lambda$.
We can therefore proceed by induction with respect to the partial order $\preceq$.  The base case is trivial, and it suffices to prove
$$\log_q |\phi_\lambda(x)| \le \beta\log_q |\phi_\lambda(1)| + o(1).$$

As the inner product of $\phi_\lambda$ with itself is bounded above independent of $q$, we have
$$\log_q |\phi_\lambda(x)| \le \frac{\log_q |C(x)|}{2}+o(1) \le \frac{\beta^2n^2}{2(1+\beta)}+o(1).$$
Thus, by (\ref{nine}) and (\ref{ten}), $\lambda_1 > n/2$.
Now $\phi_\lambda(x)$ counts the number of $x$-stable flags (\ref{flag})
with $\dim V_i/V_{i-1}$ giving the parts of the partition $\lambda$ in some given order.  We choose an order such that 
$\dim \F_q^n/V_{r-1} = \lambda_1$.  
Each such flag determines the combinatorial data of the multiplicities of the various eigenvalues of $x$ on each $V_i$.
This combinatorial data fixes an irreducible component of the variety of $x$-stable flags.  The number of possibilities for the
data is bounded independent of $q$, and each component is a product of flag varieties on spaces of dimensions 
$\le \frac{\beta^2 n}{1+\beta}$ whose dimensions add up to $n-\lambda_1$.  The dimension of a flag variety on a space of dimension $a$ is less than $a^2/2$, and since the eigenspaces of $x$ all have dimension $\le \frac{\beta^2 n}{1+\beta}$, we deduce that
$$\log_q | \phi_\lambda(x) | \le \frac{\beta^2(n-\lambda_1)n}{2(1+\beta)} + o(1).$$
On the other hand,
\begin{align*}
\log_q \phi_\lambda(1) &= \frac{n^2-\sum_{i=1}^r \lambda_i^2}{2} + o(1) \ge \lambda_1(n-\lambda_1)+o(1) \\
                                           &\ge \frac{(n-\lambda_1)n}2+ o(1),
\end{align*}
and since $\beta < 1$, this implies $\log_q |\phi_\lambda(x)| < \beta \log_q \phi_\lambda(1)$
for all $q$ sufficiently large.
\end{proof}

\bibliography{refs}
\bibliographystyle{alpha}

\end{document}